\documentclass[11pt]{article}
\usepackage{amsmath}
\usepackage{amssymb}
\usepackage{amsfonts}
\usepackage{eucal}
\newtheorem{theorem}{Theorem}[section]
\newtheorem{proposition}[theorem]{Proposition}

\newtheorem{corollary}[theorem]{Corollary}

\newtheorem{lemma}[theorem]{Lemma}
\newtheorem{remark}[theorem]{Remark}

\newenvironment{proof}[1][Proof]{\textbf{#1.} }{\ \rule{0.5em}{0.5em}}
\input pictex.sty

 \usepackage{xcolor}
 
 \usepackage{amsmath,color,ulem}
\newcommand{\stkout}[1]{\ifmmode\text{\sout{\ensuremath{#1}}}\else\sout{#1}\fi}

\begin{document}

 \title{Convergence  for weighted sums of L\"uroth type random variables  }
 	\author{Rita Giuliano \footnote{Dipartimento di
		Matematica, Universit\`a di Pisa, Largo Bruno
		Pontecorvo 5, I-56127 Pisa, Italy (email: rita.giuliano@unipi.it)}~~and Milto Hadjikyriakou\footnote{School of Sciences, University of Central Lancashire, Cyprus campus, 12-14 University Avenue, Pyla, 7080 Larnaka, Cyprus (email:
		mhadjikyriakou@uclan.ac.uk).}} 	
 \maketitle	

\begin{abstract}
 In this work we prove an asymptotic result, that under some conditions on the involved distribution functions, is valid for any Oppenheim expansion,  extending a classical result proven by W. Vervaat in 1972 for denominators of the L\"uroth case. Furthermore, we study the convergence in distribution of weighted sums of a sequence of independent random variables.    Although the  result   is of its own interest,  in the present setting  it is used   to prove   convergence in distribution   of specific  sequences of random variables   generalizing known results obtained for  L\"uroth  random variables. 
\end{abstract}


\section{Introduction}

In 1883 J. L\"uroth  (\cite{L1883}) showed that every real number $x \in (0,1]$ admits the following series expansion:
\begin{equation*}\label{espansione}
x =\frac{1}{d_1}+ \frac{1}{(s_1)d_2 }+ \frac{1}{(s_1s_2)d_3 }+ \cdots + \frac{1}{(s_1\cdots  s_n)d_{n+1}}+\cdots = \sum_{k=1}^\infty \frac{1}{\left(\prod_{h=1}^{k-1}s_h\right)d_k} ,
\end{equation*}
where $d_j= d_j(x)$ is a sequence of  integers $\geqslant 2$ and $s_n = d_n(d_n-1)$, $ n \geqslant 1$. The digits $d_n(x)$   can be viewed as random variables on the probability space $(\Omega, \mathcal{A}, P)= \big([0,1], \mathcal{B}([0,1]), \mathbb{L}\big)$, where $ \mathcal{B}([0,1])$ (where $ \mathcal{B}([0,1])$ is the $\sigma$-algebra of Borel subsets of $[0,1]$ and $\mathbb{L}=$ Lebesgue measure in $[0,1]$)  and as such they are    denoted   with capital letters $D_n$.  

\medskip

\noindent
Define
$$D^{( 1)}_n = \frac{\sum_{k=1}^nD _k }{n}.$$The following two results are  well known (see \cite {G1976}, pp. 67--68 for the first one, and \cite{V1972}, p.117 or \cite {G1976}, pp. 67--68 for the second one):
\begin{theorem}\sl\label{convergenzainprobabilita}
	As $n \to \infty$
	$$\frac{ D^{( 1)}_n}{  \log n}\mathop {\longrightarrow }^P 1.$$
\end{theorem}

\begin{theorem}\sl\label{convergenzainlegge}
	As $n \to \infty$ 
	$$ D^{( 1)}_n-\log n-1\mathop {\longrightarrow }^\mathcal{L} \mu,$$
	where $\mu$ is the probability law on $[0,1]$ determined by the characteristic function
	$$\psi(t) =\exp \left(-\frac{1}{2}\pi |t|-it\log|t|\right).$$
\end{theorem}
A similar result to Theorem \ref{convergenzainlegge} concerning the sequence of digits $\{A_n\}_{n \geq 1}$ of the continued fraction expansion of  an irrational number $x\in (0,1)$  is the one that follows, proved by P. L\'evy in \cite{L1952}.

\begin{theorem}\sl\label{Levy}$$\frac{1}{n}\sum_{k=1}^n A_k -\frac{\log n}{\log 2}\mathop {\longrightarrow }^\mathcal{L} \mu,$$
	where $\mu$ is the probability law on $[0,1]$ determined by the characteristic function
	$$\exp \left(-\frac{\pi |t|}{2 \log 2 }  - \frac{i t \log |t|}{\log 2}- \frac{i\gamma t}{\log 2}\right)$$
and  $\gamma = 0.577 \dots$ is the Euler-Mascheroni constant.
\end{theorem} 
 An estimate for the convergence described above can also be found in the literature (see Theorem 2 in \cite{H1987}). 

\medskip
\noindent
In the case of the L\"uroth sequence $\{D_n\}_{n \geq 1}$, Theorems \ref{convergenzainprobabilita} and \ref{convergenzainlegge} have been extended in \cite{G2016} by considering sums of the type
$$\sum_{k=1}^n a_{k,n}D_n,$$
where $\{a_{k,n}\}_{n \geq 1\atop k\leq n}$ is an array of positive numbers satisfying a suitable set of  assumptions. The particular class of $\{a_{k,n}\}_{n \geq 1\atop k\leq n}$ allows to apply the results of \cite{G2016} to the $r${-iterated $\alpha$-weighted  means}  of $\{D_n\}_{n \geq 1}$, i.e. to sequences built as follows: for $\alpha <1$,   the $r${-iterated $\alpha$-weighted  means  of $D_n$ are defined inductively by
\begin{equation}\label{equazione10}
D^{(\alpha,0)}_n = D_n; \qquad D^{(\alpha,r+1)}_n = \frac{\sum_{k=1}^nw_k D^{(\alpha,r)}_k }{W_n}, \quad r=0,1,2, \dots 
\end{equation}
where  
	$$w_k= \frac{1}{k^\alpha}, \qquad W_n= \sum_{k=1}^nw_k= \sum_{k=1}^n \frac{1}{k^\alpha}.$$}

\noindent
This remark is of interest also for the purposes of the present paper, since   the same kind of array $\{a_{k,n}\}_{n \geq 1\atop k\leq n}$ is used here.

\medskip
\noindent
For the  L\"uroth sequence, notice that the variables $\{D_n\}_{n \geq 1}$ are independent and identically distributed with discrete law
$$P(D_n =k) =\frac{1}{k(k-1)}, \qquad k=2, 3, \dots $$
 whereas its continuous analogue  is the law with density (w.r.t. Lebesgue measure)
$$f(x) = \frac{1}{x^2}, \qquad x\geq 1,$$
i.e. the probability density function  of the reciprocal of a random variable with uniform law on $(0,1)$.

\bigskip
\noindent
Motivated by this observation, in the present paper we first consider an independent sequence $ \{U_n\}_{n \geq 1 } $ such that, for every $n$, $U_n$ has distribution $F_n$ on $(0,1)$ (not necessarily  absolutely con\-ti\-nuous) and study the sequences of random variables
$$\sum_{k=1}^n a_{k,n}Y_n,$$
where $Y_n = \frac{1}{U_n},$ and $\{a_{k,n}\}_{n \geq 1\atop k\leq n}$ is an array of positive numbers satisfying the same assumptions as before.

\medskip
\noindent
On the sequence $\{F_n\}_{n \geq 1 }$ we impose reasonable hypotheses in order to  ensure that the analogues (or better the extensions) of Theorems \ref{convergenzainprobabilita} and \ref{convergenzainlegge} hold. This approach is fruitful since  our hypotheses  are useful also for another reason: they lead to a better understanding of the origin of the constants that appear in the limit laws. The two generalizations concerning $Y_n$ are Theorem \ref{1} and Corollary \ref{Cor1} respectively.

\medskip
\noindent
As it is well known (see \cite{G2018}) it may happen that the  results obtained for $\{Y_n\}_{n \geq 1 }$ give rise to identical results for any sequence of {\it Oppenheim expansions} $\{R_n\}_ {n \geq 1}$ suitably connected with $\{Y_n\}_ {n \geq 1 }$. This is achieved via a useful trick, presented  in  \cite{G2018} and reported here as Theorem \ref{teoremadistanza}. Oppenheim expansions are fully described in \cite{G2018}  and we recall here the most famous among them: L\"uroth (\cite{L1883}, \cite{G1976}) Engel and Sylvester series expansions (\cite{ERS1958}, \cite{G1976}); and Engel(\cite{KW2004}) and Sylvester (\cite{FWW2007}) continued fraction expansions.

\medskip
\noindent
The above mentioned phenomenon  i.e., the ability of moving from $Y_n$ to $R_n$, turns out to occur in what concerns the extension of Theorem \ref{convergenzainprobabilita}, thus for $\{R_n\}_{n \geq 1 }$ we obtain Theorem \ref{2} as a  direct  consequence of Theorem \ref{1}, which is a convergence result for the sequence $Y_n$.   We stress the fact that this kind of results are completely new for \lq\lq general\rq\rq \ Oppenheim expansions,  and only some particular cases   are  available in the literature; see \cite{G2018} and the references therein.
 However, the trick of  using  Theorem \ref{teoremadistanza} to derive Theorem \ref{2}, doesn't work if we look for the extension of Theorem \ref{convergenzainlegge}.  In this case  things are more complicated,  and a partial extension of this result has been obtained for sequences of independent random variables, by applying a completely different method based on a general result (Theorem \ref{generalres}). We point out that the case in which  the  $\{R_n\}_{n \geq 1 } $ are not independent is still to be investigated. 
 
 \medskip
 \noindent
 It is important to highlight here that the novelty of this work concerns not only the convergence results themselves but also the proposed way of generating the involved constants.

\medskip
\noindent
In closing this introduction it is worth pointing out that the extensions of Theorem \ref{convergenzainlegge}  obtained in the present work are in the spirit of {\it  weighted exact laws}  as the ones proved in \cite{A2012}, \cite{AM2018} and more recently in \cite{CGH2019}, \cite{A2019} and \cite{A2020}.  We recall that an exact law is a convergence result for a sequence $\{Z_n\}_{n \geq 1}$  in which a  suitable array  of real numbers $\{c_{k,n}\}_{n \geq 1, \atop{k\leq n}}$ ensures that
\[
\sum_{k = 1}^{n}c_{k,n} Z_k \to 1
\]
in probability ({\it  weak exact law}) or a.s. ({\it strong exact law}).

\medskip
\noindent 
Throughout the paper, the following notation will be used:
\begin{itemize}
	\item[(a)] By $a_n \sim b_n$, $n \to \infty$, (resp. $f(t) \sim g(t)$, $t \to 0$) we mean that $\lim_{n \to \infty}\frac{a_n}{b_n}=1$ (resp.  $\lim_{t \to 0}\frac{f(t)}{g(t)}=1$).
	\item[(b)] By $a_n \approx b_n$, $n \to \infty$ (resp.$f(t) \approx g(t)$, $t \to 0$), we mean that $\lim_{n \to \infty} a_n-b_n =0$ (resp. $\lim_{t \to 0} f(t) -g(t)=0$).
	\item[(c)] The symbols $C$, $M$, $c$ that appear in various cases  may represent  different constants in each appearance.	
\end{itemize}

\medskip
\noindent
 The paper is organized as follows: Section 2 contains some preliminaries; in particular the set of conditions imposed on $\{F_n\}_{n \geq 1}$ is described and a number of convergence results related to this family of distributions are presented. Section 3 contains the announced result concerning Oppenheim expansions  which leads to a series of Corollaries that can be considered as extensions and generalizations of known results that appear in \cite{G2018}. Section 4 is split into two subsections: Subsection 4.1 is devoted to the general  asymptotic  result for independent random variables already mentioned above, while Subsection 4.2 describes two  of its  applications,  which are also of independent interest. In Section 5 we attempt a discussion about the origin of the constant that appears in Proposition \ref{integralconv} and Corollary \ref{Cor1}, while the Appendix collects the proofs of some Lemmas used throughout the paper, too technical to be positioned in the body of the text.

\section{Preliminaries } 

Let $\{F_n\}_{n \geq 1}$ be a sequence of distribution functions on $[0,1]$ such that $F_n(0)=0$ for every $n$.  We shall make the following basic assumptions:
		
		\begin{itemize}
			\item[(i)] There exists a sequence of positive real numbers $\{\alpha_n\}_{n \geq 1}$ with $$0<\liminf_{n\to \infty}\alpha_n \leq \limsup_{n\to \infty}\alpha_n < \infty\quad 
			\mbox{such that}\quad
			\lim_{t\to 0^+}\sup_n\left|\frac{F_n(t)}{t}-\alpha_n\right| = 0.$$

			\item[(ii)] The functions  $$u \mapsto \frac{1}{u}\left(\frac{F_n(u)}{u}-  \alpha_n\right)$$
			are uniformly integrable on $(0,1)$, i.e. 
			\begin{equation} \label{uniformintegrability}
			\lim_{t \to 0} \sup_n \int_0^t \frac{1}{u}\left|\frac{F_n(u)}{u}-  \alpha_n\right|{\rm d}u =0.
			\end{equation}
	
	\end{itemize}

\noindent \begin{remark} \rm With regards to condition (i), 
 note that it is possible to have distributions for which $\displaystyle \limsup_{n\to \infty} \alpha_n = \infty$ or $=0$. For instance, consider the family of distributions given by
\[
F_n(t) = 
\begin{cases}
0, & t<0\\
\displaystyle\frac{c_nt}{1-c_n t}, & \displaystyle0\leq t<\frac{1}{2c_n}\\
1, & t\geq\displaystyle \frac{1}{2c_n}.
\end{cases}
\] 
If the sequence $\{c_n\}_{n \geq 1}$ is such that $\displaystyle\lim_{n\to \infty} c_n= \infty$, then  $\displaystyle \limsup_{n\to \infty} \alpha_n  = \limsup_{n\to \infty} \lim_{t\to 0^+}\frac{F_n(t) }{t} =\infty$. 
 
 \noindent
On the other hand, the family of distributions
\[
F_n(t) = 
\begin{cases}
0, & t<0\\
\displaystyle\frac{c_nt}{1-t}, & \displaystyle0\leq t<\frac{1}{1+c_n}\\
1, & t\geq\displaystyle\frac{1}{1+c_n}
\end{cases}
\] 
for which $\displaystyle\lim_{n\to \infty} c_n=  0$ leads to $\displaystyle \liminf_{n\to \infty} \alpha_n  = \liminf_{n\to \infty} \lim_{t\to 0^+}\frac{F_n(t) }{t} =0$.
\end{remark}
  
\medskip
\begin{remark} \rm Concerning condition (ii), it is easy to see that in the special case       of  $F_n$ being  differentiable i.e. $F_n^\prime =f_n$, then integration by parts leads to 
$$\int_0^1\frac{1}{u}\left(\frac{F_n(u)}{u}-  \alpha_n\right)\, {\rm d}u=\alpha_n -1+ \int_0^1\frac{f_n (u) - \alpha_n}{u}\, {\rm d}u.$$
\end{remark}

\bigskip
\bigskip
\noindent Throughout the paper we denote
$$b_{F_n}:= \int_0^1 \frac{1}{u}\left(\frac{F_n(u)}{u}-  \alpha_n\right)\, {\rm d}u \qquad \mbox{and} \qquad c_{F_n}= 1 -\alpha_n\gamma  +   b_{F_n},$$
where $\gamma$ is Euler's constant.  We also define 
\begin{equation}\label{formaalternativa}
A_n(t) = \int_{t}^{\infty}(\cos v-1){\rm d} F_n\left(\dfrac{t}{v}\right)\quad\mbox{and}\quad B_n(t) = \int_{t}^{\infty}(\sin v){\rm d} F_n\left(\dfrac{t}{v}\right).
\end{equation}
Note that in case $F_n$ has a density $f_n$, the above expressions can be written as
\begin{equation*} 
A_n(t) = t\int_{t}^{\infty}\dfrac{\cos v-1}{v^2}f_n\left(\dfrac{t}{v}\right) \, {\rm d}v\quad\mbox{and}\quad B_n(t) = t\int_{t}^{\infty}\dfrac{\sin v}{v^2}f_n\left(\dfrac{t}{v}\right) \, {\rm d}v.
\end{equation*}

\noindent Recall that a family of functions $\{h_n: \mathcal{I}\to\mathbb{R}, \, n\in\mathbb{N} \}$ is said to be bounded for $x\in \mathcal{I}$ uniformly in $n$, if there is a finite positive constant $M$ such that
\[
\sup_{n\in \mathbb{N}\atop x\in\mathcal{I}} |h_n(x)| = M.
\]
\begin{lemma} \sl 
	\label{bounds}Let $\{F_n\}_{ \geq 1}$ be a family of distribution functions such that condition (i) 
	hold. Then,
	\begin{enumerate}
		\item [(a)] The functions $t \mapsto \frac{F_n(t)}{t}$ are  bounded for $t \in (0, 1]$,  uniformly in $n$.
		\item [(b)] If moreover, condition (ii) 
		holds, the functions $t \mapsto\frac{A_n(t)}{t}$ are  bounded for $|t|<1 $,  uniformly in $n$.
	\end{enumerate}
\end{lemma}
\begin{proof}
For part (a), first observe that by the assumptions, we have that $t \mapsto \frac{F_n(t)}{t}$ are bounded uniformly in a neighborhood of $0$, i.e. there exists $c \in (0, 1]$ such that 
	$$\sup_{n\in \mathbb{N}\atop 0<t<c} \frac{F_n(t)}{t} =:M < \infty.$$
Now, for $t \in [c,1]$ we have $$\frac{F_n(t)}{t} \leq \frac{1}{c},$$
thus,
$$\sup_{n\in \mathbb{N}\atop 0<t\leq 1}\frac{F_n(t)}{t} \leq \max \left\{M, \frac{1}{c}\right\}.$$

\noindent For part (b), it suffices to treat the case $t>0;$ by applying integration by parts and using part (a) we have 
\begin{align*}
&
\left|A_n(t)\right|=\left|\int_{0}^{1}\left(\cos\left(\dfrac{t}{u}\right)-1\right){\rm d} F_n(u)\right|\leq \int_{0}^{1}\left(1-\cos\left(\dfrac{t}{u}\right)\right){\rm d} F_n(u)
\\
& =F_n(u) \left(1-\cos \left(\dfrac{t}{u}\right)  \right)\Big|_0^1 +t \int_0^1 \frac{F_n(u)}{u^2} \sin\left(\dfrac{t}{u}\right) \, {\rm d} u\\
&= (1-\cos t)- \lim_{u \to 0}\frac{F_n(u)}{u} \left(1- \cos \left(\dfrac{t}{u}\right)  \right)u+t \int_0^1 \frac{F_n(u)}{u^2} \sin\left(\dfrac{t}{u}\right) \, {\rm d} u\\
& =(1-\cos t)+t \int_0^1 \frac{F_n(u)}{u^2} \sin\left(\dfrac{t}{u}\right) \, {\rm d} u.
\end{align*} 
Note that the integral term can be equivalently written as
\begin{align*}
\int_0^1 \frac{F_n(u)}{u^2} \sin\left(\dfrac{t}{u}\right) \, {\rm d} u =\int_0^1 \frac{1}{u}\left(\frac{F_n(u)}{u}-\alpha_n \right)\sin\left(\dfrac{t}{u}\right) \, {\rm d} u +\int_0^1 \frac{\alpha_n}{u } \sin\left(\dfrac{t}{u}\right) \, {\rm d} u.
\end{align*}
For the first integral we have that 
$$\left|\int_0^1 \frac{1}{u}\left(\frac{F_n(u)}{u}-\alpha_n \right)\sin\left(\dfrac{t}{u}\right) \, {\rm d} u\right|\leq \int_0^1 \frac{1}{u}\left|\frac{F_n(u)}{u}-\alpha_n \right|  \, {\rm d} u \leq \sup_n \int_0^1 \frac{1}{u}\left|\frac{F_n(u)}{u}-\alpha_n \right|  \, {\rm d} u < \infty,$$
by 
(a).
For the second integral, by using the substitution $v = \frac{t}{u}$ we obtain
$$\int_0^1 \frac{\alpha_n}{u } \sin\left(\dfrac{t}{u}\right) \, {\rm d} u = \alpha_n \int_t^\infty\frac{v}{t} (\sin v) \frac{t}{v^2}\, {\rm d} v = \alpha_n \int_t^\infty\frac{\sin v}{v}\, {\rm d} v \leq \sup_n \alpha_n \int_0^\infty\frac{\sin v}{v}\, {\rm d} v= C \frac{\pi}{2}$$
which is true whenever $|t|<\pi$. Thus,
$$\left|A_n(t)\right|\leq (1-\cos t)+C |t|< \frac{t^2}{2}+C  |t|< C |t|, \quad\mbox{for}\quad |t| < 1 .$$
\end{proof}

\medskip
\noindent
The next two results state the uniform convergence of a particular sequence of functions of $t$  written in terms $F_n$, as $t \to 0$.
\begin{proposition}\sl
	\label{integralconv}Let $\{F_n\}_{n \geq 1}$   be  a family of distribution functions for which conditions   (i) 
	  and 
 (ii) are satisfied. Then for any integer $n$, 
	$$\lim_{t \to 0 } \left(\frac{1}{t}  \int_0^1 \left(\sin \frac{t}{u}\right){\rm d}F_n(u)  + \alpha_n \log|t|\right) =c_{F_{n}}. $$	
Furthermore, the above convergence is uniform in $n$. 	
\end{proposition}
\begin{proof}
It suffices to calculate the requested limit as $t \to 0^+$.
	For every integer $n$ we have, 
	\begin{align}&
	\nonumber \frac{1}{t}  \int_0^1 \left(\sin \frac{t}{u}\right){\rm d}F_n(u)  + \alpha_n\log t = \frac{1}{t}  \int_0^t \left(\sin \frac{t}{u}\right){\rm d}F_n(u)  +\frac{1}{t}  \int_t^1 \left(\sin \frac{t}{u}\right){\rm d}F_n(u)  + \alpha_n \log t\\
	& \label{eq1}=\underbrace{ \frac{1}{t}  \int_0^t \left(\sin \frac{t}{u}\right){\rm d}F_n(u) }_{=\Gamma_n(t)} +\underbrace{\left(\frac{1}{t}  \int_t^1 \left(\sin \frac{t}{u}\right){\rm d}F_n(u)- \int_t^1 \frac{\alpha_n}{u}\, {\rm d}u \right)}_{= \Delta_n(t)}.
	\end{align}
	We start with studying the term $ \Delta_n(t)$.
	\begin{align}&
	\nonumber  \Delta_n(t)= \frac{1}{t}  \int_t^1 \left(\sin \frac{t}{u}\right){\rm d}F_n(u)- \int_t^1 \frac{\alpha_n}{u}\, {\rm d}u \\
	&\nonumber= \frac{F_n(u) \sin  \frac{t}{u} }{t}\Big|_t^1 - \frac{1}{t} \int_t^1  F_n(u)\left( \cos \frac{t}{u}\right)\left(- \frac{t}{u^2}\right)\, {\rm d}u- \int_t^1 \frac{\alpha_n}{u}\, {\rm d}u \\
	&\nonumber= F_n(1) \frac{\sin t}{t}- \frac{F_n(t)}{t} \sin 1+  \int_t^1 \frac{F_n(u)}{u^2}\left( \cos \frac{t}{u}\right)\, {\rm d}u - \int_t^1 \frac{\alpha_n}{u}\, {\rm d}u \\
	&\label{eq2}= \underbrace{ \frac{\sin t}{t}- \frac{F_n(t)}{t} \sin 1+ \int_t^1\left(\frac{F_n(u)}{u^2} -  \frac{\alpha_n}{u} \right)\, {\rm d}u}_{=C_n(t)} +   \underbrace{\int_t^1\frac{F_n(u)}{u^2} \left(\cos \frac{t}{u} -1\right)\, {\rm d}u}_{=D_n(t)} .
	\end{align}
	For $C_n(t)$ we have that 
	\begin{equation}\label{equazione11}
	\lim_{t \to 0^+} C_n(t)= 1 -\alpha_n \sin 1 + \int_0^1 \frac{1}{u}\left(\frac{F_n(u)}{u}-  \alpha_n\right)\, {\rm d}u.
	\end{equation}
	Note that condition (i) 
	ensures that the above convergence is uniform in $n$.
	For the term $D_n(t)$ we first write 
	\begin{align*}&
	D_n(t)= \int_t^1\frac{F_n(u)}{u^2} \left( \cos \frac{t}{u} -1\right)\, {\rm d}u= \int_1^t \frac{F_n(\frac{t}{v})}{\frac{t^2}{v^2}} (\cos v-1) \left(-\frac{t}{v^2}\right) \, {\rm d}v= \int_t^1 \frac{F_n(\frac{t}{v})}{\frac{t}{v}} \cdot \frac{\cos v-1}{v}\, {\rm d}v.
\end{align*}	
 Then
$$\lim_{t \to 0^+} D_n(t) = \alpha_n \int_0^1 \frac{\cos v -1}{v}\, {\rm d}v$$
uniformly in $n$, where the convergence  follows by applying the bounded convergence theorem due to Lemma \ref{bounds} (a) and the fact that $ f(v) = \frac{\cos v -1}{v}$ is a bounded function. Using integration by parts, we can re-write the above integral as 
\begin{align*}&
\int_0^1 \frac{\cos v -1}{v}\, {\rm d}v =\frac{\sin v-v}{v}\Big|_0^1 + \int_0^1 \frac{\sin v-v}{v^2}\, {\rm d}v \\& =
(\sin 1 - 1)- \underbrace{\lim_{v \to 0}\frac{\sin v-v}{v}}_{=0} + \int_0^1 \frac{\sin v-v}{v^2}\, {\rm d}v= (\sin 1 - 1)+ \int_0^1 \frac{\sin v-v}{v^2}\, {\rm d}v.
\end{align*}
and hence
\begin{equation}
 \label{eq3}\lim_{t \to 0} D_n(t)= \alpha_n(\sin 1 - 1)+ \alpha_n \int_0^1 \frac{\sin v-v}{v^2}\, {\rm d}v. 
\end{equation}
 Equations \eqref{eq2}-\eqref{eq3} lead to  
$$
\lim_{t \to 0^+}  \Delta_n(t)=1 -\alpha_n + \alpha_n \int_0^1 \frac{\sin v-v}{v^2}\, {\rm d}v+   \int_0^1 \frac{1}{u}\left(\frac{F_n(u)}{u}-  \alpha_n\right)\, {\rm d}u,\quad \mbox{uniformly in}\quad n .
$$
Next, we study the term $ \Gamma_n(t)$.  Using integration by parts we have  
\begin{eqnarray}
\nonumber  \Gamma_n(t) &=&  \frac{1}{t}  \int_0^t \left(\sin \frac{t}{u}\right){\rm d}F_n(u) = \frac{F_n(u) \sin \frac{t}{u}}{t}\Big|_0^t-  \int_0^t \frac{F_n(u)}{t}  \left(\cos \frac{t}{u}\right)\left(- \frac{t}{u^2}\right)\, {\rm d}u\\
\nonumber &=& \frac{F_n(t)}{t}\sin 1 -\lim_{u \to 0} \frac{F_n(u) \sin \frac{t}{u}}{t} + \int_0^t \frac{F_n(u)}{u^2} \cos \frac{t}{u} \, {\rm d}u \\
\nonumber&=&
\frac{F_n(t)}{t}\sin 1 -\underbrace{\lim_{u \to 0}\frac{F_n(u)}{u} \cdot \frac{ \sin \frac{t}{u}}{\frac{t}{u}}}_{=0} + \int_0^t \frac{F_n(u)}{u^2} \cos \frac{t}{u} \, {\rm d}u \\
\label{eq4}&=&\underbrace{\frac{F_n(t)}{t}\sin 1 }_{=E_n(t)}+ \underbrace{\int_0^t \frac{F_n(u)}{u^2} \cos \frac{t}{u} \, {\rm d}u}_{={H_n}(t)} .
\end{eqnarray}
For the first term we have
$$\lim_{t \to 0^+} E_n(t) = \alpha_n \sin 1$$
uniformly in $n$ due to condition (i)
. 
For the second term we start by defining 
$$\varepsilon_n(u) : = \frac{F_n(u)}{u}- \alpha_n.$$
Then 
\begin{align*}
H_n(t) = \int_0^t \frac{F_n(u)}{u^2} \cos \frac{t}{u} \, {\rm d}u = \int_0^t \frac{1}{u} (\alpha_n + \varepsilon_n(u))\cos \frac{t}{u} \, {\rm d}u=\alpha_n \underbrace{
	\int_0^t \frac{1}{u}  \cos \frac{t}{u} \, {\rm d}u}_{=V_n(t)}+ \underbrace{\int_0^t \frac{\varepsilon_n(u)}{u}  \cos \frac{t}{u} \, {\rm d}u}_{=W_n(t)}.
\end{align*}	
Concerning the second integral, we have
$$|W_n(t)|\leq  \int_0^t \frac{|\varepsilon_n(u)|}{u}  \, {\rm d}u\to 0, \qquad t \to 0,$$
 uniformly in $n$ since by formula \eqref{uniformintegrability}, the functions $u \mapsto \frac{\varepsilon_n(u)}{u}  = \frac{1}{u}\left(\frac{F_n(u)}{u} - \alpha_n \right) $  are uniformly integrable in the neighborhood of 0.
By the change of variable $v = \frac{t}{u}$  and integration by parts, the first integral becomes
\begin{align*}&
V_n(t)=\int_ \infty^1  \frac{v}{t} \cos v \left(-\frac{t}{v^2}\right) \, {\rm d}v = \int_1^\infty\frac{\cos v}{v} \, {\rm d}v = \frac{\sin v}{v}\Big|_1^\infty + \int_1^\infty \frac{\sin v}{v^2} \, {\rm d}v \\ & = - \sin 1 + \int_1^\infty \frac{\sin v}{v^2}\, {\rm d}v;
\end{align*}
thus
 \[
\lim_{t \to 0} H_n(t) = \alpha_n\left(- \sin 1 + \int_1^\infty \frac{\sin v}{v^2}\, {\rm d}v\right)\quad\mbox{uniformly in}\quad n,
\]  
 which leads to the uniform convergence in $n$ of $ \Gamma_n(t)$ i.e. 
$$\lim_{t \to 0}  \Gamma_n(t)=\alpha_n \sin 1 -\alpha_n \sin 1 + \alpha_n \int_1^\infty \frac{\sin v}{v^2}\, {\rm d}v=\alpha_n \int_1^\infty \frac{\sin v}{v^2}\, {\rm d}v \quad \mbox{uniformly in $n$}.$$
 Finally,
$$\lim_{t  \to0}( \Gamma_n(t) +  \Delta_n(t)) =  1 -\alpha_n   + \alpha_n \left(\int_1^\infty \frac{\sin v}{v^2}\, {\rm d}v+   \int_0^1 \frac{\sin v-v}{v^2}\, {\rm d}v\right)+   \int_0^1 \frac{1}{u}\left(\frac{F_n(u)}{u}-  \alpha_n\right)\, {\rm d}u.$$
and the desired result follows by applying Lemma \ref{lemma1}.	
\end{proof}

\medskip 

The next corollary is a direct consequence of the result described in Proposition \ref{integralconv} and it generalizes Lemma 3.2 in \cite{G2018}. 

\begin{corollary} \sl \label{general1}  Let $\{F_n\}_{n \geq 1}$ be a family of distribution functions for which conditions (i) and (ii) are satisfied.
  Then we have
		$$\frac{1}{t}  \int_0^1 \left(\sin \frac{t}{u}\right){\rm d}F_n(u) \sim -\alpha_n \log |t|, \qquad t \to 0.$$
		Moreover, the above relation is uniform in $n$.		
\end{corollary}
\begin{proof}Again, we calculate the limit as $t \to 0^+$ only.
	Observe that 
		\begin{align*}
		\frac{1}{-t\alpha_n  \log t}  \int_0^1 \left(\sin \frac{t}{u}\right){\rm d}F_n(u)=\frac{1}{-\alpha_n  \log t} \left(\frac{1}{t}  \int_0^1 \left(\sin \frac{t}{u}\right){\rm d}F_n(u)+ \alpha_n  \log t- c_{F_n}\right)- \frac{c_{F_n}}{\alpha_n  \log t}+1.
		\end{align*}
		The second fraction goes to 0 as $t \to 0^+$ evidently, and Proposition \ref{integralconv} implies the same for the first fraction.  Moreover, in both terms  the convergence is  uniform in $n$: in fact, recalling assumption (i)
, we notice that the sequence $\alpha_n$ is bounded from below by a constant $C>0$, hence for the first fraction we have
		$$\left|\frac{ \frac{1}{t}  \int_0^1 \left(\sin \frac{t}{u}\right){\rm d}F_n(u)+ \alpha_n  \log t- c_{F_n} }{   \alpha_n  \log t }\right|\leq 
		\frac{\left|\frac{1}{t}  \int_0^1 \left(\sin \frac{t}{u}\right){\rm d}F_n(u)+ \alpha_n  \log t- c_{F_n}\right|}{  C | \log t|}.$$
		 By  Proposition \ref{integralconv},  the term on the right goes to 0 uniformly in $n$. 
		Similar arguments can be applied to prove that the convergence of the second term is also uniform in $n$. More precisely, 
		$$\left|\frac{c_{F_n}}{\alpha_n \log t}\right|=  \left|\frac{1}{\alpha_n \log t} - \frac{\gamma}{\log t}+  \frac{b_{F_n}}{\alpha_n \log t} \right| \leq \frac{1}{C |\log t|} + \frac{\gamma}{|\log t|} + \frac{|b_{F_n}|}{C |\log t|},$$
		 and the uniform convergence follows by assumption (i).
\end{proof}

\section{An exact weak law for Oppenheim expansions }

 Let $\{ \Theta_n\}_{n\geq  1}$ be a sequence of integer valued random variables defined on $(\Omega, \mathcal{A}, P)$, where $\Omega =[0,1]$, $\mathcal{A}$ is the $\sigma$-algebra of the Borel subsets of $[0,1]$ and $P$ is the Lebesgue measure on $[0,1]$. Let $\{F_n, n\geq 1\}$ be a sequence of probability distribution functions defined on $[0,1]$ with  $F_n(0)=0$, $\forall n$ and moreover let $\varphi_n:\mathbb{N}^*\to \mathbb{R}^+$ be a sequence of function. Furthermore, let
 $\{q_n\}_{n\geq  1}$ with $q_n=q_n(h_1, \dots, h_n)$ be a sequence of nonnegative numbers (i.e. possibly depending on the $n$ integers $h_1, \dots, h_n$) such that, for $h_1 \geq  1$ and $h_j\geq  \varphi_{j-1}(h_{j-1})$, $j=2, \dots, n$ we have
\begin{equation*}\label{densitacondizionale}
P\big( \Theta_{n+1}=h_{n+1}| \Theta_{n}=h_{n}, \dots,  \Theta_{1}=h_{1}\big)= F_n(\beta_n)-F_n(\alpha_n),
\end{equation*}
where  
\begin{equation*}
\alpha_n=\delta_n(h_n, h_{n+1}+1, q_n)  ,\quad \beta_n=\delta_n(h_n, h_{n+1}, q_n)\quad\mbox{with}\quad\delta_j(h,k,q) = \frac{ \varphi_j (h )(1+q )}{k+\varphi_j (h ) q }.
\end{equation*} 
Let $Q_n= q_n ( \Theta_1, \dots, \Theta_n)$ and define
\begin{equation}
\label{Rdef}R_{n}= \frac{  \Theta_{n+1}+\varphi_n(\Theta_n) Q_n}{\varphi_n( \Theta_n)(1+Q_n) }= \frac{1}{\delta_n( \Theta_n, \Theta_{n+1}, Q_n)}.
\end{equation}
Particular instances of this scheme are studied in \cite{L1883}, \cite{G1976} (L\"uroth series), \cite{S1974}, \cite{E1913} (Engel series), \cite{P1960} (Sylvester series), \cite{HKS2002}  (Engel continued fraction expansions).

\medskip

We define for every integer $n$, a random variable $U_n$   with distribution function $F_n$.   We assume that the variables $\{U_n\}_{n\geq  1}$ are  independent.  Furthermore, let $\displaystyle Y_n = \frac{1}{U_n}$ for every integer $n$. The characteristic function of $Y_n$ is given by
$$\psi_n(t) = \int_0^1 {\rm e} ^{{\rm i}\frac{t}{u}}  {\rm d} F_n(u),$$
which can be equivalently written as 
$$\psi_n(t) = 1 + A_n(t) + {\rm i}B_n(t),$$
where
\begin{equation}
\label{defAB}A_n(t) = \int_{0}^{1}\left(\cos\left(\dfrac{t}{u}\right)-1\right){\rm d} F_n(u)\quad\mbox{and}\quad B_n(t) = \int_{0}^{1}\sin\left(\dfrac{t}{u}\right){\rm d} F_n(u).
\end{equation}
If $v= t/u$, the quantities $A_n(t)$ and $B_n(t)$ can be written as the expressions in \eqref{formaalternativa}.

\medskip

\noindent The result that follows is  instrumental  for Theorem \ref{1}  and  its proof can be found in \cite{G2018}.

\begin{lemma} \label{linearizzazione}\sl Let $c \in(0,1)$ be fixed. There exists a constant $ m(c)$ such that, for $|z|< c$, we have
	$$\big|\log(1+z)-z\big|\leqslant m (c)|z|^2.$$	
\end{lemma}

The theorem that follows is a convergence result for the sequence of random variables $Y_n$ defined above. Its proof is motivated by the proof of Theorem 3 in \cite{G2016} which is a convergence result for the Luroth random variables.

\begin{theorem}\label{1}
	\sl
	Let $\{a_{k,n}\}_{n \geq 1\atop k\leq n}$ be an array of positive  numbers such that for some sequence $\{\rho_n\}_{n \geq 1}$ with $\displaystyle\lim_{n \to \infty}\rho_n \log n = + \infty$, we have
	\begin{align}&\label{x}
	\limsup_{n \to \infty }\frac{\sum_{k=1}^n \alpha_k a_{k,n} | \log (\alpha_ka_{k,n})|  }{ \rho_n\log n}< \infty, \quad   \quad \lim_{n \to \infty }\frac{\sum_{k=1}^n\alpha_k a_{k,n}  \log (\alpha_k a_{k,n} ) }{ \rho_n\log n} \,\,\hbox{ exists} =:-\ell ;\\&\label{y} 
	\hbox{   and    the    sequence} \quad n \mapsto \sum_{k=1}^n \alpha_k a_{k,n} \, \,\hbox{is bounded. }   
	\end{align}  
	Moreover, assume that for the family of distributions $\{F_n\}_n$ conditions (i) and (ii) are satisfied. Then 
	$$T_n:= \frac{1}{\rho_n\log n}\sum_{k = 1}^{n}a_{k,n} Y_k\mathop{\longrightarrow}^P \ell , \qquad n \to \infty.$$
\end{theorem} 

\begin{remark}\rm \label{r1} Note that because of condition (i),
	\begin{enumerate}
	 \item[(a)]$ \displaystyle n \mapsto \sum_{k=1}^n \alpha_k a_{k,n} $ is bounded if and only if  $\displaystyle n \mapsto \sum_{k=1}^n  a_{k,n} $ is. 
	 
	 \item[(b)]Similarly, $ \displaystyle\frac{\sum_{k=1}^n\alpha_k a_{k,n}  \log (\alpha_k a_{k,n} ) }{\rho_n \log n} \to-\ell $  if and only if $ \displaystyle\frac{\sum_{k=1}^n \alpha_k a_{k,n}  \log (  a_{k,n} ) }{ \rho_n\log n} \to-\ell.$
	 \end{enumerate}
 \end{remark}

 \noindent \begin{proof}
	Let $\psi_{T_n}$ be the characteristic function of $T_n$.
	We use Levy's Theorem and prove that
	$$\lim_{n \to \infty}\log\psi_{T_n}(t) ={\rm i}t  \ell \qquad \forall \, t\in \mathbb{R}.$$
	Denote for simplicity
	$$c_{k,n}:=\frac{ a_{k,n}}{\rho_n\log n}, \qquad k=1, \dots, n  .$$
 In the sequel, the expression $\log\psi_{T_n}(t) $ denotes the analytic continuation of $\log z$ along the path $(\psi_{T_n}(t))_{t\in \mathbb{R}}$ starting from $\log\psi_{T_n}(0) =0$ under the assumption that $\psi_{T_n}(t)$ is a characteristic function without real zeros. Then,

\[
\log\psi_{T_n}(t)= \sum_{k=1}^n \log E\left[{\rm e}^{{\rm i}t c_{k,n}Y_k } \right]= \sum_{k=1}^n\log \psi_k(tc_{k,n})=\sum_{k=1}^n  \log\left(1+A_k(tc_{k,n})+{\rm i}B_k(tc_{k,n})\right).
\]
	
	\noindent
By Lemma \ref{linearizzazione}, it suffices to prove that
	\begin{itemize}
		\item[(a)] $$\sum_{k=1}^n A_k(tc_{k,n})\to 0, \qquad n \to \infty;$$
		\item[(b)]$$\sum_{k=1}^n B_k(tc_{k,n})\to t \ell, \qquad n \to \infty;$$
		\item[(c)]$$\sum_{k=1}^n\left(A_k^2(tc_{k,n})+ B_k^2(tc_{k,n})\right)\to 0, \qquad n \to \infty.$$		
	\end{itemize}
Assume first that $t>0$.  

\medskip

\noindent (a)  Let  $ S:=\sup_n\sum_{k=1}^n \alpha_{k}a_{k,n}$. Notice that assumption
(i) ensures that 
$$c_{k,n} = \frac{a_{k,n}}{\rho_n\log n}=\frac{\alpha_{k}a_{k,n}}{\alpha_{k}\rho_n\log n} \leq \frac{S}{C\rho_n\log n},$$
so there exists   $n_0$  such that, for $n > n_0$ and $k =1, \dots, n,$ we have $c_{k,n }< 1$.
By Lemma \eqref{bounds} (ii)  we have
\begin{align*}&
\left|\sum_{k=1}^{n}A_k(tc_{k,n})\right| \leq \sum_{k=1}^{n}C|t c_{k,n}|\leq \frac{C t}{\rho_n\log n}\sum_{k=1}^{n}a_{k,n}\to  0, \qquad n \to \infty
\end{align*} 	
because of 
the first part of Remark \ref{r1}. 

\medskip

\noindent (b)  Due to the uniform convergence described in Corollary \ref{general1}, we have that $\forall\epsilon>0$ there exists $c\in (0,1]$ such that, for all $k$ and for all $t<c$
\[ 
1-\epsilon<-\dfrac{1}{\alpha_kt\log t}\int_{0}^{1}\left(\sin \frac{t}{u}\right){\rm d}F_k\left(u\right)<1+\epsilon.
\]
By employing assumption (i) again, let $C>0$ be such that $\alpha_k >C$ for every $k$ and let $n_0\in\mathbb{N}$ such that $ \forall n>n_0 $, $ \displaystyle\frac{tS}{C\rho_n \log n}< c$. Then,
\[
0<tc_{k,n}=\frac{t\alpha_ka_{k,n}}{\alpha_k\rho_n\log n}\leq \dfrac{tS}{C\rho_n\log n}<c,
\]
which ensures that, for every $k=1, \dots, n$,
\[ 
1-\epsilon<-\dfrac{1}{\alpha_kc_{k,n}t\log (tc_{k,n})}\int_{0}^{1}\left(\sin \frac{tc_{k,n}}{u}\right){\rm d}F_k\left(u\right)<1+\epsilon.
\]
Recall that $tc_{k,n}< c\in(0,1]$ and therefore the logarithm is negative. Hence, for $n>n_0$,
\[ 
-(1-\epsilon)\sum_{k=1}^{n}\alpha_kc_{k,n}t\log (tc_{k,n})<\sum_{k=1}^{n}\int_{0}^{1}\left(\sin \frac{tc_{k,n}}{u}\right){\rm d}F_k\left(u\right)<- (1+\epsilon)\sum_{k=1}^{n}\alpha_kc_{k,n}t\log (tc_{k,n}),
\]
which means that
\[
\lim_{n\to\infty}\dfrac{1}{\sum_{k=1}^{n}\alpha_kc_{k,n}t\log (tc_{k,n})}\sum_{k=1}^{n}\int_{0}^{1}\left(\sin \frac{tc_{k,n}}{u}\right){\rm d}F_k\left(u\right)= -1.
\]
Then,
\begin{eqnarray*}
	\sum_{k=1}^{n}B_k(tc_{k,n}) &=& \sum_{k=1}^{n}\int_{0}^{1}\left(\sin \frac{tc_{k,n}}{u}\right){\rm d}F_k\left(u\right)\\
	&\sim&- t\sum_{k=1}^{n}\alpha_kc_{k,n}\log (tc_{k,n})\\
	&=&-  t\log t\sum_{k=1}^{n}\alpha_kc_{k,n}-  t\sum_{k=1}^{n}\alpha_kc_{k,n}\log c_{k,n}.
\end{eqnarray*}
Observe that the first term converges to zero as $n$ tends to infinity while for the second term we have that
\[
\sum_{k=1}^{n}\alpha_kc_{k,n}\log c_{k,n} = \frac{1}{\rho_n\log n}\sum_{k=1}^{n}\alpha_ka_{k,n}\log a_{k,n}-\frac{\log(\rho_n\log n)}{\rho_n\log n}\sum_{k=1}^{n}\alpha_ka_{k,n} \to -\ell,\qquad n\to\infty,
\]
by Remark \ref{r1}.
Thus,
\[  
\sum_{k=1}^{n}B_k(tc_{k,n}) \to  t \ell ,\qquad n\to\infty.
\]
(c) For the third convergence first observe that 
\[
\sum_{k=1}^{n}A^2_k(tc_{k,n}) \leq \left(\sum_{k=1}^{n}A_k(tc_{k,n}) \right)^2\to 0, \qquad n\to\infty.
\]
Next, with an argument analogous to the one used for $\sum_{k=1}^{n}B_ k(tc_{k,n})$, we write
\[
	\sum_{k=1}^{n}B^2_k(tc_{k,n}) \sim t^2\sum_{k=1}^{n}\alpha_k^2c_{k,n}^2\log^2(tc_{k,n}), \qquad n \to \infty.
\]
For the latter summation we have that
\[
\sum_{k=1}^{n}\alpha_k^2c_{k,n}^2\log^2(tc_{k,n})\leq \dfrac{2(\log t)^2}{(\rho_n\log n)^2}\sum_{k=1}^{n} \alpha_k^2a_{k,n}^2 +\dfrac{2}{(\rho_n\log n)^2}\sum_{k=1}^{n} \alpha_k^2a_{k,n}^2\log^2c_{k,n}.
\]
Note that for the first term we have
\[
\dfrac{1}{( \rho_n\log n)^2}\sum_{k=1}^{n} \alpha_k^2 a_{k,n}^2 \leq \left(\dfrac{1}{\rho_n\log n}\sum_{k=1}^{n}\alpha_k a_{k,n}\right)^2\to 0,\qquad n\to \infty
\]
while for the second term we have
\begin{eqnarray*}
	\sum_{k=1}^{n} \dfrac{ \alpha_k^2a^2_{k,n}}{\rho_n^2\log^2 n}\log^2c_{k,n}&\leq& 2\sum_{k=1}^{n}\dfrac{\alpha_k^2a^2_{k,n}}{\rho_n^2\log^2 n}(\log^2a_{k,n}+\log^2(\rho_n\log n)) \\
	&=&\dfrac{2}{\rho_n^2\log^2 n}\sum_{k=1}^{n}\alpha_k^2a^2_{k,n}\log^2a_{k,n}+\dfrac{2\log^2(\rho_n\log n)}{\rho_n^2\log^2 n}\sum_{k=1}^{n}\alpha_k^2a_{k,n}^2.
\end{eqnarray*}
Note that 
\[
\dfrac{\log^2(\rho_n\log n)}{\rho^2_n\log^2 n}\sum_{k=1}^{n}\alpha_k^2a_{k,n}^2 \leq \dfrac{(\log(\rho_n\log n))^2}{(\rho_n\log n)^2}\left(\sum_{k=1}^{n}\alpha_ka_{k,n}\right)^2\leq \left(\dfrac{S\log(\rho_n\log n)}{\rho_n\log n}\right)^2\to 0,\qquad n\to\infty.
\]
By assumption \eqref{y}, we have for every $k,n$  
$$ a_{k,n}= \frac{\alpha_k a_{k,n}}{\alpha_k }\leqslant \frac{ \sum_{k=1}^n \alpha_ka_{k,n}}{C} \leqslant \frac{S}{C};$$  
hence, from  the inequality
$$\sup_{0\leqslant x \leqslant \frac{S}{C}} x  |\log x|< \infty,$$
we deduce
\begin{align}\label{mm}
0 \leqslant  \frac{\sum_{k=1}^n{\alpha_k^2 a^2_{k,n}\log^2 a _{k,n}}}{\rho_n^2\log^2 n}\leqslant C\frac{\sum_{k=1}^n a _{k,n}|\log  a _{k,n}|}{ \rho_n^2\log^2 n}\to 0, \qquad n \to \infty,
\end{align}
by  assumptions \eqref{x} and  (i).

\medskip

\noindent Consider now the case $t<0$. Starting again from the calculation of $\psi_{T_n}$, we see that
\begin{align*}&
\psi_{T_n}(t) = 1+\int_0^1 \left(\cos \frac{t}{u}-1\right) {\rm d}F_n(u) +{\rm i} \int_0^1 \left(\sin \frac{t}{u}\right) {\rm d}F_n(u)\\
&=1+\int_0^1 \left(\cos \left(-\frac{t}{u}\right)-1\right){\rm d}F_n(u)  -{\rm i} \int_0^1 \sin \left(-\frac{t}{u}\right) {\rm d}F_n(u)= 1+A_n(|t|)-{\rm i}B_n(|t|).
\end{align*}
Hence, by the same arguments as in case $t>0$  and observing that there is a minus sign before $B_n(|t|)$, we get again$${\rm (a)} \lim_{n\to \infty}\sum_{k=1}^n A_k(tc_{k,n}) =0;\quad{\rm (b)} \lim_{n\to \infty} - \sum_{k=1}^n B_k(tc_{k,n})   =-|t| \ell =t\ell;\quad {\rm (c)}\lim_{n\to \infty}\sum_{k=1}^n\big(A_k^2(tc_{k,n})+ B_k^2(tc_{k,n})) =0,$$
which concludes the proof for the case $t<0$ and of the Theorem.
\end{proof}

\medskip

\noindent For the proof of Theorem \ref{2} we recall the following result from \cite{CGH2019}, which allows the transition from the random variables $Y_n$ to general Oppenheim expansions.

\begin{theorem}\sl \label{teoremadistanza}  Let $\{R_n\}_{n\geq 1}$ be as in \eqref{Rdef} and let $U_1,\ldots, U_n$ be independent random variables such that $U_n\sim F_n$ for any integer $n$. Let $\phi_{R_1, \dots, R_n}$ be the characteristic function of the vector $(R_1, \dots, R_n)$ and let $\psi_n$ be the characteristic function of the random variable defined as $Y_n = U_n^{-1}$ for every $n$. Then, for every $  (t_{1,n}, \dots, t_{n,n})\in \mathbb{R}^n$ and $n \geqslant 1$ we have
	$$\left| \phi_{R_1, \dots, R_n}(t_{1,n }, \dots, t_{n,n })-\prod_{k=1}^n\psi_k(t_{k,n })\right|\leqslant  \sum_{k=1}^n |t_ {k,n}|.$$
\end{theorem}

 \begin{theorem} \label{2}
	\sl  
	In the same assumptions as in Theorem \ref{1}  and Theorem \ref{teoremadistanza}, we have
	$$\frac{1}{\rho_n\log n}\sum_{k = 1}^{n}a_{k,n} R_k\mathop{\longrightarrow}^P\ell, \qquad n \to \infty.$$
\end{theorem} 
\begin{proof}
	 Define $W_n:= \frac{1}{\rho_n\log n}\sum_{k = 1}^{n}a_{k,n} R_k$ and let $\phi_{W_n}$ be its characteristic function.  Then,
	$$\phi_{W_n}(t)= E\left[e^{it \sum_{k=1}^n c_{k,n}R_k } \right]=
	\phi_{R_1, \dots, R_n}(t_ {{1,n} }, \dots, t_{{n,n}} )$$
	with $t_{k,n}=   tc_{k,n}$.  Recall that by assumption (i)
	$$tc_{k,n} = \frac{t\alpha_ka_{k,n}}{\alpha_k \rho_n\log n}\leq \frac{t\alpha_ka_{k,n}}{C\rho_n\log n}.$$ 
   Thus, by applying Theorem \ref{teoremadistanza} we have that
	$$\left| \phi_{R_1, \dots, R_n}(t_ {k,1}, \dots, t_{k,n} )-\prod_{k=1}^n\psi_k(t_ {k,n})\right|\leqslant  \sum_{k=1}^n |t_ {k,n}|=  |t|\frac{\sum_{k=1}^n  \alpha_k a_{k,n} }{\rho_n\log n} \to 0, \qquad n \to \infty,$$
	where the convergence is established due to \eqref{y} . The desired result follows from Theorem \ref{1}.
\end{proof}

\begin{remark} \rm
It is important to highlight that Theorem \ref{2} is valid for any Oppenheim expansion for which the involved distribution functions satisfy the desired conditions. 
\end{remark}

\noindent As a direct consequence of Theorem \ref{2} a series of weak laws can be obtained.  Particularly,   weak exact laws  are  proven in the next four Corollaries for the L\"uroth, the Engel and the Sylvester sequences respectively. 
 
\begin{corollary} \sl
	\label{Luroth}Let $\{D_n\}_{n \geq 1}$ be the L\"uroth sequence (see the Introduction). Under the assumptions of Theorem \ref{1} and Theorem \ref{teoremadistanza}, we have 
	\begin{equation*}
	\frac{1}{\rho_n\log n}\sum_{k=1}^{n} a_{k,n}D_{k+1} \mathop{\longrightarrow}^P \ell.
	\end{equation*}
\end{corollary}
\begin{proof} It is known  (see \cite{G2018}) that 
	 $R_n = D_{n+1} -1$ falls into the scheme described  by \eqref{Rdef} for $\varphi_n(h) =0$ and $Q_n = 0$. According to Theorem \ref{2} 
	\[
	\frac{1}{\rho_n \log n}\sum_{k=1}^{n} a_{k,n}(D_{k+1}-1) \mathop{\longrightarrow}^P \ell, \quad n\to \infty,
	\]
	which can be equivalently written as
	\[
	\frac{1}{ \rho_n\log n}\sum_{k=1}^{n} a_{k,n}D_{k+1}-\frac{1}{\rho_n\log n}\sum_{k=1}^{n} a_{k,n} \mathop{\longrightarrow}^P \ell,\qquad n\to \infty. 
	\]
	Note that due to Remark \ref{r1}, the second term converges to zero, which leads to the desired conclusion.	
\end{proof}

\begin{corollary}\sl
	\label{Lurothit} Let $\{D^{(\alpha,1)}_n\}_{n \geq 1}$   be the sequence  defined in \eqref{equazione10} (for a given $\alpha < 1$). Under the assumptions of Theorem \ref{1} and Theorem \ref{teoremadistanza}, we have 
	\begin{equation*}
	\frac{1}{\rho_n\log n}D_n^{(\alpha,1)} \mathop{\longrightarrow}^P \ell.
	\end{equation*}
\end{corollary}
\begin{proof}
Consider the sequence $a_{k,n}  = w_k/W_n$, $1\leq k \leq n$, where $w_k = k^{-\alpha}$ and $W_n = \sum_{k=1}^{n}w_k$. Then,
	\begin{align*}
	&\sum_{k=2}^{n}a_{k,n}D_k = \sum_{k=1}^{n-1}a_{k+1,n}D_{k+1}=\sum_{k=1}^{n-1}a_{k+1,n}(R_k+1) =\sum_{k=1}^{n-1}a_{k+1,n} R_k + \sum_{k=1}^{n-1}a_{k+1,n}.
	\end{align*}
	Observe that 
	\begin{align*}
	&\frac{1}{\rho_n\log n}\sum_{k=1}^{n-1}a_{k+1,n} R_k+\frac{1}{\rho_n\log n}\sum_{k=1}^{n-1}a_{k+1,n} = \frac{1}{\rho_n\log n}\sum_{k=1}^{n-1}\frac{a_{k+1,n}}{a_{k,n}} a_{k,n}R_k+\frac{1}{\rho_n\log n}\sum_{k=1}^{n-1}a_{k+1,n}\\
	&\leq \frac{1}{\rho_n\log n}\sum_{k=1}^{n-1} a_{k,n}R_k+\frac{1}{\rho_n\log n}\sum_{k=1}^{n-1}a_{k+1,n} \mathop{\longrightarrow}^P \ell,\qquad n\to \infty
	\end{align*}
	due to Theorem \ref{2}. Recall that $D_n^{(\alpha,1)} = \sum_{k=1}^{n}a_{k,n}D_k$ and the result follows immediately.
\end{proof}

 \begin{corollary} \sl
	\label{Engel}	  Let $\{D_n\}_{n \geq 1}$ be the Engel's sequence (see \cite{G2018}, Section 5). Under the assumptions of Theorem \ref{1} and Theorem \ref{teoremadistanza}, and the additional condition 
	\begin{equation}\label{equazione7}
	  \sup_n\max_{k \leq n}a_{k,n}< \infty
	\end{equation}
we have 
	\begin{equation*}
	\frac{1}{ \rho_n\log n}\sum_{k=1}^{n} a_{k,n}\frac{D_{k+1}}{D_k} \mathop{\longrightarrow}^P \ell.
	\end{equation*}
 \end{corollary} 
\begin{proof} 
Let $\displaystyle R_n = \dfrac{D_{n+1} -1}{D_n-1}$. It is known (see \cite{G2018}) that 
	 $R_n$ falls into the scheme described in \eqref{Rdef}  with $\varphi_n (h)=h$ and $Q_n=0$. According to Theorem 3.4 
	\[
	\frac{1}{ \rho_n\log n}\sum_{k=1}^{n} a_{k,n}\frac{D_{k+1}-1}{D_k-1} \mathop{\longrightarrow}^P \ell,\quad n\to\infty.
	\]
Moreover, it follows from $D_{k+1} \geq D_k $ (see \cite{G2018}) that  
$$\dfrac{D_{k+1}  }{D_k }\leq  \dfrac{D_{k+1}-1 }{D_k -1},$$	
and this inequality  ensures that
\begin{equation}\label{equazione6}
\limsup_{n \to \infty} \frac{1}{\rho_n\log n}\sum_{k=1}^{n} a_{k,n}\dfrac{D_{k+1}  }{ D_k  } \leq \lim_{n \to \infty} \frac{1}{ \rho_n\log n}\sum_{k=1}^{n} a_{k,n}\dfrac{D_{k+1} -1}{ D_k -1}= \ell.
\end{equation}

\noindent On the other hand, as it is proved in  \cite{G2018},   we also have
	
	$$\dfrac{D_{n+1} -1}{D_n-1}\sim  \dfrac{D_{n+1} }{D_n}$$
	thus, for every $ 0<\epsilon <1$ there exists $n_0$ such that, for every $k > n_0$ we have
	$$  \dfrac{D_{k+1} }{D_k}\geq  (1-\epsilon)\dfrac{D_{n+1} -1}{D_k-1} $$
	and, as a consequence, for every $n > n_0$
	$$ \frac{1}{  \rho_n\log n}\sum_{k=n_0+1}^{n}a_{k,n}\dfrac{D_{k+1} }{D_k}\geq(1-\epsilon) \frac{1}{\rho_n\log n}\sum_{k=n_0+1}^{n}a_{k,n}  \dfrac{D_{n+1} -1}{D_k-1}. $$
	Notice that
	 $$0 \leq \frac{1}{ \rho_n\log n}	\sum_{k=  1}^{n_0}a_{k,n}\dfrac{D_{k+1} }{D_k}\leq \frac{\sup_n\max_{k \leq n}a_{k,n}}{\rho_n\log n}\sum_{k=  1}^{n_0} \dfrac{D_{k+1}  }{D_k}\to 0,\qquad n \to \infty, $$ 
	 by assumption \eqref{equazione7};   by the same argument we have also
	 $$\frac{1}{ \rho_n\log n}	\sum_{k=  1}^{n_0}a_{k,n}\dfrac{D_{k+1}-1 }{D_k-1}\to 0,\qquad n \to \infty.$$
We deduce that
$$\liminf_{n \to \infty} \frac{1}{  \rho_n\log n}\sum_{k=1}^{n} a_{k,n}\dfrac{D_{k+1}  }{ D_k  } \geq (1-\epsilon)\lim_{n \to \infty} \frac{1}{ \rho_n\log n}\sum_{k=1}^{n} a_{k,n}\dfrac{D_{k+1} -1}{ D_k -1}= \ell$$ 	
and by the arbitrariness of $\epsilon$ we 
have $$\liminf_{n \to \infty} \frac{1}{\rho_n\log n}\sum_{k=1}^{n} a_{k,n}\dfrac{D_{k+1}  }{ D_k  } \geq \ell, $$
which, together with \eqref{equazione6}, concludes the proof.	
\end{proof}
  
\begin{corollary} \sl
	\label{Sylvester} Let $\{D_n\}_{n \geq 1}$ be the Sylvester's sequence (see \cite{G2018}, Section 5). Under the assumptions of Theorem \ref{1} and Theorem \ref{teoremadistanza} and condition \eqref{equazione7}
	\[
	\frac{1}{\rho_n\log n}\sum_{k=1}^{n} a_{k,n}\frac{D_{k+1}}{D_k^2} \mathop{\longrightarrow}^P \ell,\quad n\to\infty.
	\]
\end{corollary}
\begin{proof}
Let $\displaystyle R_n = \dfrac{D_{n+1} -1}{D_n(D_n-1)}$. It is known (see \cite{G2018}) that 
	$R_n  $ falls into the scheme described by \eqref{Rdef} for $\varphi_n (h)=h(h-1)$ and $Q_n=0$.  According to Theorem 3.4 
\[
\frac{1}{ \rho_n\log n}\sum_{k=1}^{n} a_{k,n}\frac{D_{k+1}-1}{D_k(D_k-1)} \mathop{\longrightarrow}^P \ell,\quad n\to\infty.
\]
Note that $R_n \geq 1$ and this leads to $D_{n+1}\geq D_n^2-D_n +1 \geq D_n$. Thus
\[
\frac{D_{n+1}}{D_n^2}\leq \frac{D_{n+1}-1}{D_n(D_n-1)},
\]
which ensures that 
\begin{equation}\label{equazione8}
\limsup_{n \to \infty} \frac{1}{  \rho_n\log n}\sum_{k=1}^{n} a_{k,n}\dfrac{D_{k+1}  }{ D_k^2  } \leq \lim_{n \to \infty} \frac{1}{  \rho_n\log n}\sum_{k=1}^{n} a_{k,n}\dfrac{D_{k+1} -1}{ D_k(D_k -1)}= \ell.
\end{equation}
Furthermore,  as it is proved in  \cite{G2018},
$$\dfrac{D_{n+1} -1}{D_n(D_n-1)}\sim  \dfrac{D_{n+1} }{D_n^2}.$$
The rest of the argument runs as in the proof of Corollary \ref{Engel}.
\end{proof}

\begin{remark}\rm
	For the case $\rho_n \equiv 1$, $a_{k,n} = n^{-1}$, $k \leq n$ and $\alpha_n \equiv 1$ for all $n$, Corollaries \ref{Luroth}, \ref{Engel} and \ref{Sylvester} reduce  to Theorems 5.1, 5.2 and 5.3 in \cite{G2018} respectively. 
\end{remark}	

\section{Convergence in distribution for independent random variables}

The main result of this section is Theorem \ref{generalres}, where the convergence in distribution of weighted sums of a sequence of independent random variables is proven. The result is then used to obtain convergence results for specific sequences of particular interest.

\subsection{A general result for independent random variables}
 \begin{theorem} \sl 
	\label{generalres} Let $\{Z_n\}_{n \geq 1}$ be a sequence of { independent} random variables. For every $n$ define 
	\begin{equation*} \label{equazione3}
	g_{n} (t) = \frac{\phi_{n}(t) -{\rm e}^{{\rm i}t } }{{\rm e}^{{\rm i}t }-1},
	\end{equation*}
	where $\phi_{n}(t):=\phi_{Z_n}(t)= E\left[{\rm e}^{{\rm i}t Z_n}\right]$. Assume that
	\begin{itemize}
		\item[(i)] 
		\begin{equation}
		\label{as1} \limsup_{t \to 0}\frac{\sup_n |g_{n}(t)| }{ \left | \log |t |\right |^\eta}< \infty
		\,\, \qquad \hbox{for some } \eta >0;
		\end{equation}
		\item[(ii)] there exist two sequences $ 0\leq\{c_{1,n}\}_{n\geq 1}  $   and $\{c_{2,n}\}_{n\geq 1}$ such that  
			\begin{equation}
			\label{as2} 
		 \lim_{t\to 0} \sup_n  \big|   g_n(t) +c_{1,n}\log (1- {\rm e}^{{\rm i}t }\big) -c_{2,n}\big|=0. 
			\end{equation}
		
	\end{itemize}
	Furthermore, let $\{a_{k,n}\}_{n \geq 1\atop k\leq n}$ be an array of positive  numbers such that
	\begin{equation}   
	\lim_{n \to \infty}  \sum_{k=1}^n a_{k,n} \,\,  \hbox{exists} =\kappa ;\label{equazione1} \end{equation}  
	\begin{equation} \label{equazione2}
	\lim_{n \to \infty} m_n  =0, \qquad \hbox{ where} \qquad  m_n:=\max_{1\leqslant k \leqslant n}a_{k,n};
	\end{equation} 
	 
	\begin{equation} \label{equazione4}
	\lim_{n \to \infty}  \sum_{k=1}^n a_{k,n}c_{1,k} \,\,  \hbox{exists} =\ell.
	\end{equation} 
Let     $$V_n:= \sum_{k=1}^n a_{k,n}Z_k  -\left(\kappa +\sum_{k=1}^n a_{k,n}c_{2,k}\right)  +\sum_{k=1}^na_{k,n}c_{1,k}\log a_{k,n}.$$  
Then 
\begin{equation}
\label{res}V_n\mathop {\longrightarrow }^\mathcal{L} \mu,   \qquad n \to \infty  ,
\end{equation}
where $\mu$ is the probability law on $[0,1]$ determined by the characteristic function
	$$\xi(t) =\exp \left(-\frac{\pi}{2}\ell|t|-{\rm i}t \ell\log|t| \right).$$
In particular, if $\lim_{n \to \infty} c_{1,n}= c_1$, then   \eqref{res} holds true,  where $\mu$ is the probability law on $[0,1]$ determined by the characteristic function
	$$\xi(t) =\exp \left(-\frac{\pi}{2}\kappa c_1|t|-{\rm i}t c_1\kappa\log|t| \right).$$
\end{theorem}
\begin{proof}
The second statement of the Theorem follows immediately from the first one; therefore we present in detail only the proof of the first part.  Using L\'evy's Continuity Theorem   and under the assumption that $\phi_{V_n}(t)$ is a characteristic function without real zeros we have 
	\begin{align*}
	&\log \left(\phi_{V_n}(t)\right) \\
	&=
	\log \left(\prod_{k=1}^n E\left[e^{{\rm i}t  a_{k,n} Z_k }\right] \right)-{\rm i}t  \left(\kappa+\sum_{k=1}^na_{k,n}c_{2,k}\right) + {\rm i}t \sum_{k=1}^na_{k,n} c_{1,k} \log a_{k,n}\\
	&=
	\sum_{k=1}^n \log E\left[e^{{\rm i}t  a_{k,n} Z_k }\right]   -{\rm i}t  \left(\kappa+\sum_{k=1}^na_{k,n}c_{2,k}\right) +{\rm i}t \sum_{k=1}^na_{k,n} c_{1,k} \log a_{k,n}\\
	& =
	\sum_{k=1}^n \log\left(\phi_{k} (t  a_{k,n})\right) -{\rm i}t  \left(\kappa+\sum_{k=1}^na_{k,n}c_{2,k}\right) + {\rm i}t \sum_{k=1}^na_{k,n} c_{1,k}\log a_{k,n} ,
	\end{align*}
	and we show that
	$$ \sum_{k=1}^n \log\left(\phi_{ k}  (t  a_{k,n})\right) -{\rm i}t  \left(\kappa+\sum_{k=1}^na_{k,n}c_{2,k}\right) + {\rm i}t \sum_{k=1}^n c_{1,k}a_{k,n}\log a_{k,n} +{\rm i} t \ell  \log |t|+ |t| \ell \frac{\pi}{2}  \to  0,\qquad n \to \infty.$$
	Using  Lemma \ref{approx1}, we see that it suffices to calculate the limit of
	\begin{align*}&
	\sum_{k=1}^n \left(   \phi_{k}  (ta_{k,n})-1\right) -{\rm i}t  \left(\kappa+\sum_{k=1}^na_{k,n}c_{2,k}\right)   + {\rm i}t  \sum_{k=1}^na_{k,n}  c_{1,k} \log a_{k,n}  +{\rm i} t \ell  \log |t|+ |t| \ell \frac{\pi}{2} \\
	& =\sum_{k=1}^n \left((e^{{\rm i}ta_{k,n}}-1)  +(e^{{\rm i}ta_{k,n}}-1) g_{k} \left(ta_{k,n} \right) \right)-{\rm i}t  \left(\kappa+\sum_{k=1}^na_{k,n}c_{2,k}\right)
	+{\rm i}t   \sum_{k=1}^na_{k,n}  c_{1,k}\log a_{k,n} +{\rm i} t \ell  \log |t|+ |t| \ell \frac{\pi}{2}\\
	&=I_1 + I_2+I_3+I_4,
	\end{align*}
	where 
	\begin{itemize}
		\item[~] $$I_1 = \sum_{k=1}^n (e^{{\rm i}ta_{k,n}}-1) -{\rm i}t \kappa ,\qquad I_2 =\sum_{k=1}^n(e^{{\rm i}ta_{k,n}}-1-{\rm i}t a_{k,n}) g_k\left( ta_{k,n}  \right),$$
		\item[~]$$I_3 ={\rm i}t \sum_{k=1}^na_{k,n}\left( g_{k}\left(ta_{k,n} \right) +   c_{1,k}  \log (1-e^{{\rm i}ta_{k,n}})- c_{2,k}\right) 
		\qquad
		\mbox{and}$$ \item[~]$$
		I_4 = {\rm i}t   \sum_{k=1}^na_{k,n} c_{1,k} \left(\log a_{k,n}- \log \left(1-e^{{\rm i}ta_{k,n}} \right)\right)  +{\rm i} t \ell  \log |t|+|t| \ell \frac{\pi}{2}.$$
	\end{itemize}
	Observe that $I_1$ goes to 0 since we can write it as 
	\begin{align*}
	\sum_{k=1}^n (e^{{\rm i}ta_{k,n}}-1) -{\rm i}t \kappa = \sum_{k=1}^n  (e^{{\rm i}ta_{k,n}}-1 - {\rm i}t a_{k,n}) +{\rm i}t \left(\sum_{k=1}^n a_{k,n} - \kappa\right).
	\end{align*}
	The second summand goes to 0 by assumption \eqref{equazione1}, while, by {Lemma \ref{approx2}}, 
	$$\sum_{k=1}^n  (e^{{\rm i}ta_{k,n}}-1 - {\rm i}t a_{k,n}) \approx \frac{1}{2}\sum_{k=1}^n a^2_{k,n}\leq \frac{m_n}{2}\sum_{k=1}^n a_{k,n}\leq \frac{S}{2}m_n \to 0,$$
	due to \eqref{equazione2} for $S:= \sup_n \sum_{k=1}^na_{k,n}.$ As for $I_4$, using the  Lemma \ref{approx3} , we get, for sufficiently large $n$,
	\begin{align*}&
		{\rm i}t\sum_{k=1}^na_{k,n} c_{1,k} \left(\log a_{k,n}- \log \left(1-e^{{\rm i}ta_{k,n}} \right)\right)+  {\rm i}t \ell \log |t| + |t|\ell \frac{\pi}{2}\\
		&\approx {\rm i}t\sum_{k=1}^na_{k,n} c_{1,k} \left(\log a_{k,n}-\log \left(-{\rm i}t a_{k,n}\right)\right)+   {\rm i}t \ell \log |t| + |t|\ell \frac{\pi}{2}
		\\&= {\rm i}t\sum_{k=1}^na_{k,n} c_{1,k} \left(\log a_{k,n}-\log \left|t a_{k,n}\right|+ +{\rm i}\frac{\pi}{2}{\rm sign  }\, t\right)+ {\rm i}t \ell \log |t| + |t|\ell \frac{\pi}{2}\\
		& =  {\rm i}t\sum_{k=1}^na_{k,n} c_{1,k} \left( -\log |t|   +{\rm i}\frac{\pi}{2}{\rm sign  }\, t \right)+ {\rm i}t \ell \log |t| + |t|\ell \frac{\pi}{2}\\
	&= - {\rm i}t\log |t|\left(\sum_{k=1}^na_{k,n} c_{1,k}-\ell\right)	 -t 	 {\rm sign  }\, t\frac{\pi}{2}\sum_{k=1}^na_{k,n}c_{1,k}+ |t|\ell \frac{\pi}{2} 	\\
	&=- {\rm i}t\log |t|\left(\sum_{k=1}^na_{k,n} c_{1,k}-\ell\right)	 -|t|\frac{\pi}{2}\left(\sum_{k=1}^na_{k,n} c_{1,k}-\ell\right)\to 0.
		\end{align*}
	Lemma \ref{approx4}  ensures that $I_3$ goes to $0$.  It remains to prove that $I_2$ goes to 0. This follows from the fact that $t g(t) \to 0$ as $t \to 0$  and   Lemma \ref{approx5}  since
	\begin{align*}&
	\left|\sum_{k=1}^n(e^{{\rm i}ta_{k,n}}-1-{\rm i}t a_{k,n}) g_{ k}\left( ta_{k,n}  \right)\right| \approx \frac{1}{2}\left|\sum_{k=1}^n a^2_{k,n}t^2g_{ k}\left( ta_{k,n} \right) \right|
	\leq \frac{1}{2}\sup_{k}\left|a_{k,n}t^2g_{ k}\left( ta_{k,n}  \right) \right|\sum_{k=1}^n a_{k,n}\\& \leq \frac{S}{2} \sup_{k}\left|a_{k,n}t^2g_{ k}\left( ta_{k,n}  \right)\right|.
	\end{align*}
\end{proof}

\begin{remark} \rm Denote $k_n(t) = \frac{\phi_n(t)-1}{{\rm i}t}$. Then
	$$k_n(t) =1 + g_n(t) + \big(\phi_n(t)-1\big) \frac{{\rm e}^{{\rm i}t}-1-{\rm i}t}{{\rm i}t({\rm e}^{{\rm i}t}-1)}.$$
	Thus, assuming that $\displaystyle\lim_{t \to 0}\sup_n |\phi_n(t)-1|=0$, (i) and (ii) of Theorem \ref{generalres} are verified for $k_n$ iff they hold for $g_n$ (concerning (ii) for $k_n$, the constants $c_{1,n}$ remain unaltered, while the $c_{2,n}$ are replaced by $c_{2,n}+1$).
	
	\smallskip
	\noindent
	Notice also that in our setting the limits $\displaystyle\lim_{t \to 0}g_n(t)$ and $\displaystyle\lim_{t \to 0}k_n(t)$ do not exist, which is in accordance with the fact that all the random variables considered in the present paper have infinite mean.
\end{remark}

\subsection{Applications}

\begin{corollary}  \label{Cor1}\sl
Let $U_n$ be the random variables defined in Section 3 such that their  distribution functions $F_n$ satisfy conditions (i)  
and (ii)  with  $\alpha_n \to \alpha $ as $n\to \infty$  and let $\{ a_{k,n} \}_{n \geq 1, \atop k \leq n}$ be an array of positive numbers satisfying the assumptions of Theorem \ref{generalres}. 
Then for $Y_n = U_n^{-1}$  and
$$V_n:= \sum_{k=1}^n a_{k,n}Y_k  -  \left(\kappa+\sum_{k=1}^n a_{k,n}(c_{F_k}-1) \right)   + \sum_{k=1}^n \alpha_k a_{k,n}\log a_{k,n}, $$ we have
	$$V_n\mathop {\longrightarrow }^\mathcal{L} \mu,   \qquad n \to \infty  ,$$ where $\mu$ is the probability law on $[0,1]$ determined by the characteristic function	
$$\xi(t)= \exp\left(-{\rm i}t\alpha \kappa \log|t| - \frac{\pi \alpha \kappa |t|}{2}\right).$$
\end{corollary}
\begin{proof}
The result will be obtained by applying Theorem \ref{generalres} for $Z_n =Y_n$. In this case, $\phi_n(t) = \phi_{Y_n} (t)= \psi_n(t)$, where $\psi_n (t) $ is defined as in Section 3. We will prove that both assumptions of Theorem \ref{generalres} are satisfied.

\noindent Assumption \eqref{as1} will be proven for $\eta =1$. It is sufficent to prove that
$$\limsup_{t \to 0}\sup_n\frac{1}{|\log t|}\cdot \left|\frac{ A_n(t) + {\rm i}B_n(t) }{{\rm e}^{{\rm i}t} -1 }\right|< \infty,$$
where $A_n(t)$ and $B_n(t)$ are defined in \eqref{defAB} (or \eqref{formaalternativa}). 	We have
\begin{align*}&
\left|\frac{ A_n(t) + {\rm i}B_n(t) }{{\rm e}^{{\rm i}t} -1 }\right|^2 = \frac{A_n^2(t)+ B_n^2(t)}{2 (1 -\cos t)}\sim  \frac{A_n^2(t)+ B_n^2(t)}{t^2}\\ 
&= \left(\int_t^{+\infty} \frac{1-\cos v}{v^2} {\rm d}F_n\left(\frac{t}{v}\right)\right)^2 + \left(\frac{1}{t}\int_0^{1} \sin \left(\frac{t}{u}\right){\rm d}F_n(u)\right)^2 \leq \frac{1}{ 4} + \left(\frac{1}{t}\int_0^{1} \sin \left(\frac{t}{u}\right){\rm d}F_n(u)\right)^2
\\
& \sim  \frac{1}{4}+  \alpha_n ^2 \log^2 t \sim \alpha^2 \log^2 t,\quad \mbox{as}\quad t \to 0,
\end{align*}
where the first equivalence follows from Corollary \ref{general1}.	

\medskip

\noindent Assumption \eqref{as2} will be verified for $c_ { 1,n} =    \alpha_n $ and $c_ { 2,n} =c_{F_n}-1$.  Observe that, 
$$\lim_{t \to 0}\left(\log (1 - {\rm e}^{{\rm i}t}) - \log (-{\rm i}t)\right)=0.$$
This is true since for small $t$,
$$\log (1 - {\rm e}^{{\rm i}t}) - \log (-{\rm i}t) =\log \left(\frac{1 - {\rm e}^{{\rm i}t}}{-{\rm i}t}\right)= \log  \left(\frac{  {\rm e}^{{\rm i}t}-1}{ {\rm i}t}\right)= \log \left(1+ \frac{{\rm e}^{{\rm i}t}- 1 -{\rm i}t }{{\rm i}t}\right)$$
and
$$\lim_{t \to 0}\frac{{\rm e}^{{\rm i}t}-1 -{\rm i}t }{{\rm i}t} = \lim_{t \to 0} \frac{O(t^2)}{{\rm i}t}=0.$$
Then, for  every $n$  we have that
\begin{align*}
&
g_ n (t) +   \alpha_ n \log (1- e^{it})=  -1 +\frac{  A_n(t)+i B_n(t) }{e^{it}-1} +  \alpha_n \log (1- e^{it})\approx  -1 +\frac{  A_n(t)+i B_n(t) }{e^{it}-1} +\alpha_n \log (-it)\\
&=
-1 +\frac{  A_n(t)+i B_n(t) }{e^{it}-1} +  \alpha _n \log (|t|)-i\alpha_n  \frac{\pi}{2}{\rm sign\, t}.
\end{align*}  
We assume first that $t>0$, so the latter expression becomes
$$  -1 +\frac{  A_n(t)+i B_n(t) }{e^{it}-1} +  \alpha_n \log (t)-i\alpha_n \frac{\pi}{2}.$$
The second term can be equivalently written as
\begin{align*}&
\frac{  A_n(t)+{\rm i} B_n(t) }{e^{{\rm i} t}-1}= \frac{( A_n(t)+{\rm i}  B_n(t))(e^{-{\rm i} t }-1)}{ 4 \sin^2 \frac{t}{2}}=  \frac{( A_n(t)+{\rm i}  B_n(t))((\cos t-1)- {\rm i} \sin t)}{ 4 \sin^2 \frac{t}{2}}\\&=\frac{A_n(t)(\cos t -1)+ B_n(t) \sin t }{4 \sin^2 \frac{t}{2}}+ {\rm i}  \,\,\frac{B_n(t)(\cos t -1)-A_n(t) \sin t}{4 \sin^2 \frac{t}{2}}\\&=: E_n(t) +{\rm i}  F_n(t)
\end{align*}
and therefore
\[
g_n(t) +  \alpha_n  \log (1- e^{{\rm i}t}) \approx(-1+E_n(t)+\alpha_n \log t)+{\rm i}\left(F_n(t)-\alpha_n  \frac{\pi}{2}\right).
\]
For the first term we have that
\[
-1+E_n(t)+\alpha_n \log t = -1-\frac{A_n(t)}{2}+\frac{\cos \frac{t}{2}}{2}\frac{B_n(t)}{\sin \frac{t}{2}} +\alpha_n \log t.
\]
It can easily be derived that $A_n(t) \to 0$  as $t\to 0$  by Lemma \ref{bounds}(b).
Next we study the convergence of 
$$\lim_{t \to 0} \left(\frac{\cos\frac{t}{2}}{2}\cdot \frac{B_n(t)}{\sin \frac{t}{2}}+ \alpha_n \log t\right)=
\lim_{t \to 0}  \left( \cos \frac{t}{2}\cdot \frac{\frac{t}{2}}{\sin \frac{t}{2}}\cdot \frac{1}{t}\int_0^{1} \left(\sin\frac{t}{u}\right){\rm d}F_n(u) + \alpha_n \log t\right).$$
Denote
$$S(t) =\cos \frac{t}{2}\cdot \frac{\frac{t}{2}}{\sin \frac{t}{2}}$$
and write the preceding expression as
$$S(t) \left(\frac{1}{t}\int_0^{1} \left(\sin\frac{t}{u}\right){\rm d}F_n(u) +  \alpha_n\log t\right)+  \alpha_n(1-S(t)) \log t.$$
Obviously $S(t) \to 1$ as $t \to 0$ and, from Proposition \ref{integralconv} we get
$$\lim_{t \to 0} S(t) \left(\frac{1}{t}\int_0^{1} \left(\sin\frac{t}{u}\right){\rm d}F_n(u) + \alpha_n \log t\right) =c_{F_n},$$
uniformly in  $n$. Now  we calculate the
$$\lim_{t \to 0}(1-S(t))\log t.$$
We start by writing
$$(1-S(t))\log t =\left(\log 2 + \log \frac{t}{2}\right)(1-S(t)),$$
so, putting $x= \frac{t}{2}$, it suffices to calculate  the limit 
$$\lim_{x \to 0} \log x\left( 1 - \frac{x \cos x}{\sin x}\right).$$
It can easily be proven that
$$\log x\left( 1 - \frac{x \cos x}{\sin x}\right)\sim\log x \frac{x - x \cos x}{\sin x}= \log x \cdot \frac{x - \frac{x^3}{3}- x(1 - \frac{x^2}{2})+ o (x^3)}{x + o(x^3)}= \log x \cdot \frac{\frac{x^2}{6}+ o(x^2)}{1 + \frac{o(x^3)}{x}}, $$
so
$$\lim_{x \to 0} \log x\left( 1 - \frac{x \cos x}{\sin x}\right) =0.$$

\noindent Summarizing, we have obtained that
$$\lim_{t \to 0} (-1 + E_n(t) +\alpha_n \log t)= c_{F_n}-1,$$
uniformly in $n$.   For the calculation of 
$$\lim_{t \to 0}\left(F_n(t) - \alpha_n\frac{\pi}{2}\right),$$
 notice that
\begin{align*}
&F_n(t) - \alpha_n\frac{\pi}{2}= \frac{B_n(t)(\cos t -1)- A_n(t) \sin t}{4 \sin^2 \frac{t}{2}} - \alpha_n \frac{\pi}{2}= \frac{B_n(t)(\cos t -1) }{4 \sin^2 \frac{t}{2}} - \frac{  A_n(t) \sin t}{4 \sin^2 \frac{t}{2}}  - \alpha_n \frac{\pi}{2}\\ 
&=-\frac{1}{2} B_n(t)- \frac{1}{2}
\cdot \frac{A_n(t) \cos \frac{t}{2}}{\sin \frac{t}{2}}- \alpha_n \frac{\pi}{2}.
\end{align*}
Starting with $$\lim_{t \to 0} \int_0^1 \sin \frac{t}{u}\, {\rm d} F_n(u)=0$$
and integrating by parts we have
\begin{align*}&
\int_0^1 \sin \frac{t}{u}\, {\rm d} F_n(u)= F_n(u) \sin \frac{t}{u}\Big|_0^1 + t \int_0^1 \frac{F_n(u)}{u^2}\cdot \cos \frac{t}{u} \, {\rm d}u \\&=\sin t - \lim_{u \to 0}F_n (u) \sin \frac{t}{u}+ t \int_0^1 \frac{F_n(u)}{u^2}\cdot \cos \frac{t}{u} \, {\rm d}u.
\end{align*}	
The second summand is equal to
$$t\lim_{u \to 0}  \frac{ F_n(u)}{u}  \cdot\frac{ \sin \frac{t}{u}}{ \frac{t}{u}}=0,$$
uniformly in  $n$   by assumption (ii)
. For  the functions
$$  u \mapsto\frac{F_n(u)}{u^2}\cdot \cos \frac{t}{u}  $$
we write 
\begin{align*}&
\int_0^1 \frac{F_nu)}{u^2}\cdot \cos \frac{t}{u} \, {\rm d}u= \int_0^1 \frac{1}{u}\left(\frac{F_n(u)}{u}-\alpha _n\right)\cos \frac{t}{u}\, {\rm d}u + \alpha _n\int_0^1\frac{1}{u}\cos \frac{t}{u}\, {\rm d}u.
\end{align*}	
The first integral is clearly  finite, uniformly in $n$, by assumption \eqref{uniformintegrability}; hence
$$\lim_{t \to 0}t \int_0^1 \frac{1}{u}\left(\frac{F_n(u)}{u}-\alpha _n\right)\cos \frac{t}{u}\, {\rm d}u=0,$$
uniformly in  $n$. Thus it remains to prove that
$$\lim_{t \to 0}t\int_0^1\frac{1}{u}\cos \frac{t}{u}\, {\rm d}u=0.$$ 
By the change of variable $\frac{t}{u}=z$, we have
\begin{align*}&
t\int_0^1\frac{1}{u}\cos \frac{t}{u}\, {\rm d}u= t\int_t^\infty \frac{\cos z}{z}\, {\rm d}z=- t{\rm Ci}(t) = -\gamma t- t\log t -t\int_0^t\frac{1-\cos z}{z}{\rm d}z,
\end{align*}
(recall that ${\rm Ci}(x)$ is the {\it cosine integral function}). It is easily seen that the last integral above is finite, so the conclusion is that the requested limit is equal to 0.

\medskip

\noindent Next, we calculate
$$\lim_{t \to 0}\left( - \frac{1}{2}
\cdot \frac{A_n(t) \cos \frac{t}{2}}{\sin \frac{t}{2}}- \alpha_n \frac{\pi}{2}\right).$$
 Observe that 
\[
- \frac{1}{2}
\cdot \frac{A_n(t) \cos \frac{t}{2}}{\sin \frac{t}{2}} = \cos \frac{t }{2} \cdot \frac{\frac{t}{2}}{\sin \frac{t}{2}}\cdot \int_t ^\infty \frac{1-\cos v}{v^2} \, {\rm d}F_n\left(\frac{t}{v}\right)= S(t)  \int_t ^\infty \frac{1-\cos v}{v^2} \, {\rm d}F_n\left(\frac{t}{v}\right)
\]
 and for the calculation of the integral we have 
\begin{align*}
&\int_t^{+\infty}\frac{1-\cos v}{v^2}{\rm d}F_n\left(\frac{t}{v}\right) = \frac{1}{t}\int_0^{1}\left(1 - \cos \frac{t}{u} \right){\rm d}F_n(u)\\
& = \frac{1}{t}\left[\left(1 - \cos \frac{t}{u} \right)F_n(u)\right]_{0}^{1}-\frac{1}{t}\int_0^{1}F_n(u){\rm d}\left(1 - \cos \frac{t}{u} \right)\\
& = \frac{F_n(1)(1-\cos t)}{t}-\frac{1}{t}\lim_{u\to0}F_n(u)\left(1 - \cos \frac{t}{u} \right)-\frac{1}{t}\int_{0}^{1}F_n(u)\left(\sin \frac{t}{u}\right)\left(-\frac{t}{u^2}\right)\,{\rm d}u\\
&= \frac{1-\cos t}{t}-\lim_{u\to 0}\left( \frac{F_n(u)}{u}\cdot \frac{1-\cos \frac{t}{u}}{\frac{t}{u}}\right)+\int_0^{1}\frac{F_n(u)}{u^2}\left(\sin \frac{t}{u}\right)\,{\rm d}u.
\end{align*}	
Note that the first two terms converge to 0 and therefore we only need to study the convergence of the last term. We define
\[
\epsilon_n(u) = \frac{F_n(u)}{u}-\alpha_n
\]
and now we have that
\begin{align*}
&\int_0^{1}\frac{F_n(u)}{u^2}\left(\sin \frac{t}{u}\right)\,{\rm d}u = \int_{0}^{1} \frac{\epsilon_n(u)}{u}\sin \frac{t}{u}\,{\rm d}u +\alpha_n \int_{0}^{1}\dfrac{1}{u}\sin \dfrac{t}{u}\,{\rm d}u\\
&=\int_{0}^{1} \frac{\epsilon_n(u)}{u}\sin \frac{t}{u}\,{\rm d}u + \alpha_n\int_{t}^{+\infty}\frac{\sin x}{x}\,{\rm d}x\\
&\to  \alpha_n \frac{\pi}{2}\quad \mbox{ as }\quad t\to0^+,
\end{align*}	
uniformly in  $n$  since $\int_{0}^{1} \frac{\epsilon_n(u)}{u}\sin \frac{t}{u}\,{\rm d}u$ tends to 0 uniformly in  $n$.	
Thus 
$$ - \frac{1}{2}
\cdot \frac{A_n(t) \cos \frac{t}{2}}{\sin \frac{t}{2}}  - \alpha_n \frac{\pi}{2} \to 0, \qquad t \to 0$$
uniformly in  $n$.

\noindent
Now for $t<0.$ In this case we write
\begin{align*}&g_n(t) + \alpha_n \log (1 - {\rm e}^{{\rm i}t})= -1 +\frac{  A_n(t)+{\rm i} B_n(t) }{e^{{\rm i}t}-1} +  \alpha _n \log (1- e^{{\rm i}t})= -1 + \frac{  A_n(|t|)-{\rm i} B_n(|t|) }{e^{-{\rm i}|t|}-1} +  \alpha _n \log (1- e^{{\rm i}t})\\
&\approx   -1 + \frac{  A_n (|t|)-{\rm i} B_n(|t|) }{e^{-{\rm i}|t|}-1} +  \alpha_n \log ( -  {\rm i}t)= \overline{ -1 + \frac{  A_ n(|t|)+{\rm i} B_n(|t|) }{e^{{\rm i}|t|}-1}}+ \alpha_n\log (|t|)+ {\rm i}\alpha _n  \arg (-{\rm i}t)\\&=  \overline{ -1 + \frac{  A_n(|t|)+{\rm i} B_n(|t|) }{e^{{\rm i}|t|}-1}}+\alpha_n \log (|t|)+ {\rm i}\alpha _n \frac{\pi}{2}= \overline{-1 + \frac{  A_ n (|t|)+{\rm i} B_ n (|t|) }{e^{{\rm i}|t|}-1}+\alpha _n  \log (|t|)- {\rm i}\alpha _n  \frac{\pi}{2}},
\end{align*}
which converges to 0 by the first part of this calculation (case $t>0$).	
\end{proof}	
	
\begin{corollary}\label{Cor2} \sl
	Let $\{Z_n\}_{ n \geq 1}$ be independent random variables with density
	\[
	P(Z_n = k) = p_{n,k-1}-p_{n,k}\quad\mbox{for}\quad k=2,3,\ldots
	\]
	for $\displaystyle p_{n,k} = \frac{1-\beta_n}{k-\beta_n}$,  where $\{\beta_n\}_{n \geq 1}\in [0,1)$ is a sequence bounded by a constant $c<1$; let ${ \{  a_{k,n}  \}_{n \geq 1, \atop k \leq n}}$be an array of positive numbers satisfying the assumptions of Theorem \ref{generalres}.  We assume in addition that $$ \lim_{n\to\infty} \sum_{k=1}^na_{k,n}(1-\beta_k) \mbox{ exists} :=\ell.$$ 
	Let  
	 $$V_n:= \sum_{k=1}^n a_{k,n}Z_k  -\left(\kappa +\sum_{k = 1}^{n}a_{n,k}(1-\beta_k)\int_0^1 \frac{1- x^{\beta_{k}}}{x^{\beta_{k}} (1-x)}{\rm d}x \right)  +\sum_{k=1}^na_{k,n}(1-\beta_k)\log a_{k,n}.$$   
	Then,
	$$V_n\mathop {\longrightarrow }^\mathcal{L} \mu,   \qquad n \to \infty  ,$$
	where $\mu$ is the probability law on $[0,1]$ determined by the characteristic function$$\xi(t) =\exp\left(-{\rm i}t\ell \log|t| - \frac{\pi\ell |t|}{2}\right).$$
	
\end{corollary}	
\begin{proof}
	For the characteristic function  of $Z_n$ we have that
	\begin{align*}&
	\phi_n(t)=\phi_{Z_n}(t)= E[e^{{\rm i} t Z_n}]= \sum_{k=2}^\infty e^{{\rm i}tk} (p_{ n,k-1}- p_{ n,k})   \\= &
	\sum_{k=2}^\infty  e^{{\rm i}tk}p_{ n,k-1} - \sum_{k=2}^\infty e^{{\rm i}tk} p_{ n,k}=e^{{\rm i}t}\sum_{k=2}^\infty  e^{{\rm i}t(k-1)}p_{ n,k-1}- \sum_{k=2}^\infty e^{{\rm i}tk} p_{ n,k}\\&=e^{{\rm i}t} \sum_{k=1}^\infty e^{{\rm i}tk} p_{ n,k}- \sum_{k=1}^\infty e^{{\rm i}tk} p_{ n,k} +e^{{\rm i}t}
	=
	(e^{{\rm i}t}-1)  \sum_{k=1}^\infty e^{{\rm i}tk} p_{ n,k} + e^{{\rm i}t} .
	\end{align*}
	Hence, setting
	$$h_{n}(z) = \sum_{k=1}^\infty z^{ k} p_{ n,k}, \qquad z \in \mathbb{C},$$
	which is the generating function of the sequence $ \{p_{n,k}\}_{k\geq 1}$, we obtain
	\begin{align*}&
	\phi_{n}(t)= e^{{\rm i}t}   +(e^{{\rm i}t}-1) h_{n}(e^{{\rm i}t})  \\&
	\end{align*}
	and $$g_n(t)=\frac{\phi_{n}(t)-{\rm e}^{{\rm i}t } }{{\rm e}^{{\rm i}t }-1}=    h_{ n}({\rm e}^{{\rm i}t}) =  \sum_{k=1}^\infty  {\rm e}^{{\rm i}t k} p_{ n,k}.$$
	 Note that $h_n(z)$ can be written as 
	$$h_{n}(z) =(1-\beta_{n}) \sum_{k=1}^\infty z^{ k}  \frac{1}{  k- \beta_{n}}, \qquad z \in \mathbb{C} . $$
	To this extent, notice that we can write
	$$\frac{1-\beta_{n}}{  k-\beta_{ n}} = \frac{\Gamma(2-\beta_{ n})}{\Gamma (1) \Gamma (1-\beta_{ n})}  \cdot  \frac{\Gamma(k )\Gamma(k-\beta_{ n} )}{\Gamma(k+1-\beta_{ n} )(k-1)!} ,$$
	so that
	\begin{align*}&
	\sum_{k=1}^\infty p_{ n, k} z^k =  \frac{\Gamma(2-\beta_n)}{\Gamma (1) \Gamma (1-\beta_{n})}  \sum_{k=1}^\infty \frac{\Gamma(k )\Gamma(k-\beta_{n} )}{\Gamma(k+1-\beta_{n} )(k-1)!}z^k\\&=  z \cdot  \frac{\Gamma(2-\beta_{n})}{\Gamma (1) \Gamma (1-\beta_{n})}  \sum_{k=1}^\infty \frac{\Gamma(k )\Gamma(k-\beta_{n})}{\Gamma(k+1-\beta_{n})(k-1)!}z^{k-1}= 
	z \cdot  \frac{\Gamma(2-\beta_{n})}{\Gamma (1) \Gamma (1-\beta_{n})}  \sum_{k=0}^\infty \frac{\Gamma(k+1 )\Gamma( k+1-\beta_{n} )}{\Gamma( k +2-\beta_{n})k!}z^{k }\\&=z\cdot _2\kern-1mm F_1(1,1-\beta_{n},2-\beta_{n};z),
	\end{align*}
	where $ _2\kern +0,3mm  F_1(1,1-\beta_{n},2-\beta_{n};z)$  is the {\it Gauss hypergeometric  series} (with parameters $1$, $1-\beta_{n}$ and $2-\beta_{n}$)
	by formula 15.1.1 of \cite{AS}.  It is known that the circle of convergence of this series is $|z|=1$ (see again  15.1.1 of \cite{AS}). 	It is also known (see formula 15.3.1 of \cite{AS}) that
	\begin{equation}\label{eq5}
	_2\kern-0,3mm F_1(1,1-\beta_{n},2-\beta_{n};z) = (1-\beta_{n})\int_0^1 \frac{1}{\xi^{\beta_{n}}(1-\xi z)}\,{\rm d }\xi
	\end{equation}
	 and therefore we have
		$$h_n(z)= z (1-\beta_n)\int_0^1  \frac{1}{\xi^{\beta_n}(1-\xi z)}\, {\rm d}\xi,\qquad z \in \mathbb{C}.$$ 
	
\noindent The statement of the Corollary follows immediately by Theorem \ref{generalres} as long as the assumptions of the theorem are confirmed. First, we prove assumption \eqref{as1} with $\eta =1$. 
Using formula \eqref{eq5} we write, for $t>0$,
\begin{align*}
\frac{g_n(t)}{\log t}=  \frac{{\rm e}^{{\rm i}t}(1-\beta_{n}) \int_0^1 \frac{1}{\xi^{\beta_{n}}(1-\xi {\rm e}^{{\rm i}t})}\,{\rm d }\xi }{\log t}=  \frac{ \int_0^1 \frac{1}{\xi^{\beta_{n}}(1-\xi {\rm e}^{{\rm i}t})}\,{\rm d }\xi}{    \log(1-{\rm e}^{{\rm i}t})   }(1-\beta_{n}){\rm e}^{{\rm i}t} \frac{\log(1-{\rm e}^{{\rm i}t})}{\log t}.
\end{align*}
Now, it is not difficult to check that, for every $n$,
$$\left| \int_0^1 \frac{1}{\xi^{\beta_n}(1-\xi {\rm e}^{{\rm i}t})}\,{\rm d }\xi\Big|\leq  \Big| \int_0^1 \frac{1}{\xi^{c}(1-\xi {\rm e}^{{\rm i}t})}\,{\rm d }\xi\right|.$$

\noindent
Thus
$$ \sup_n\left|\frac{ \int_0^1 \frac{1}{\xi^{\beta_n}(1-\xi {\rm e}^{{\rm i}t})}\,{\rm d }\xi}{ \log(1-{\rm e}^{{\rm i}t}) }\right|   \leq \left|\frac{  \int_0^1 \frac{1}{\xi^{c}(1-\xi {\rm e}^{{\rm i}t})}\,{\rm d }\xi}{ \log(1-{\rm e}^{{\rm i}t})} \right|$$
which converges to $1$ by \cite{WW1915}, ex. 18 p. 293. Moreover
\begin{align*}&
\frac{\log(1-{\rm e}^{{\rm i}t})}{\log t}= \frac{ \log(1-{\rm e}^{{\rm i}t})-\log (- {\rm i}t)}{ \log t}+ \frac{  \log (-{\rm i}t)}{\log t}=\frac{ \log(1-{\rm e}^{{\rm i}t})-\log (-{\rm i}t)}{ \log t}+\frac{1}{\log t} - {\rm i}\frac{\frac{\pi}{2}}{\log t} \to 0, \qquad t\to 0^+,
\end{align*}
by  \eqref{eq6}.

\noindent Next, we wish to prove that 
$$\lim_{t \to 0}\sup_n |h_n({\rm e}^{{\rm i}t}) + c_{1,n}\log(1-{\rm e}^{{\rm i}t})-c_{2,n}|=0,$$
with
$$c_{1,n}= 1-\beta_n; \qquad c_{2,n}= (1-\beta_n)\int_0^1 \frac{1-\xi^{\beta_n}}{\xi^{\beta_n}(1-\xi)}\, {\rm d}\xi.$$
Recalling formula 15.1.3 of \cite{AS}, we can write  
\begin{align*}&
h_n({\rm e}^{{\rm i}t})   + c_{1,n}\log(1-{\rm e}^{{\rm i}t})-c_{2,n}=   (1-\beta_n)\left( \int_0^1  \frac{{\rm e}^{{\rm i}t}}{\xi^{\beta_n}(1-\xi {\rm e}^{{\rm i}t})}\, {\rm d}\xi - \int_0^1 \frac{{\rm e}^{{\rm i}t}}{1-\xi {\rm e}^{{\rm i}t}}\, {\rm d}\xi-\int_0^1 \frac{1-\xi^{\beta_n}}{\xi^{\beta_n}(1-\xi)}\, {\rm d}\xi\right)\\&=
(1-\beta_n)\left(\int_0^1\frac{1- \xi^{\beta_n}}{\xi^{\beta_n}}\left(\frac{{\rm e}^{{\rm i}t}}{1-\xi {\rm e}^{{\rm i}t}}-\frac{1}{1-\xi}\right)\, {\rm d}\xi\right)=(1-\beta_n)({\rm e}^{{\rm i}t}-1)\int_0^1 \frac{1-\xi^{\beta_n}}{\xi^{\beta_n}(1-\xi {\rm e}^{{\rm i}t})(1-\xi)}\, {\rm d}\xi.
\end{align*}
So we have to prove that
$$\lim_{t \to 0}\sup_n \left|({\rm e}^{{\rm i}t}-1)\int_0^1 \frac{1-\xi^{\beta_n}}{\xi^{\beta_n}(1-\xi {\rm e}^{{\rm i}t})(1-\xi)}\, {\rm d}\xi \right|=0.$$
We have
\begin{align*}&
\left|({\rm e}^{{\rm i}t}-1)\int_0^1 \frac{1-\xi^{\beta_n}}{\xi^{\beta_n}(1-\xi {\rm e}^{{\rm i}t})(1-\xi)}\, {\rm d}\xi \right|= |{\rm e}^{{\rm i}t}-1| \left|\int_0^1 \frac{1-\xi^{\beta_n}}{1-\xi} \cdot\frac{1}{\xi^{\beta_n}(1-\xi {\rm e}^{{\rm i}t})}\, {\rm d}\xi\right|\\& \leq  |{\rm e}^{{\rm i}t}-1|\int_0^1\frac{1}{\xi^{\beta_n}|1-\xi {\rm e}^{{\rm i}t}|}\, {\rm d}\xi,
\end{align*}
since 
$$ 0\leq \frac{1-\xi^{\beta_n}}{1-\xi} \leq 1,$$
recalling that $\beta_n < 1$.
So it remains to prove that
$$\lim_{t \to 0 }\sup_n |{\rm e}^{{\rm i}t}-1|\int_0^1\frac{1}{\xi^{\beta_n}|1-\xi {\rm e}^{{\rm i}t}|}\, {\rm d}\xi =0.$$
First
\begin{align*}&
|1-\xi {\rm e}^{{\rm i}t}|^2= |1-\xi\cos t +{\rm i}\xi\sin t|^2 = (1-\xi\cos t)^2+\xi^2\sin t^2=\xi^2- 2 \xi\cos t +1.
\end{align*}
Hence
\begin{align*}&
|{\rm e}^{{\rm i}t}-1|\int_0^1\frac{1}{\xi^{\beta_n}|1-\xi {\rm e}^{{\rm i}t}|}\, {\rm d}\xi = \left|\frac{{\rm e}^{{\rm i}t}-1}{{\rm i}t}\right|\cdot|{\rm i}t|\int_0^1\frac{1}{\xi^{\beta_n} \sqrt{\xi^2- 2 \xi\cos t +1}}\, {\rm d}\xi \\&=\left|\frac{{\rm e}^{{\rm i}t}-1}{{\rm i}t}\right|\int_0^1\frac{1}{\xi^{\beta_n} }\cdot \sqrt{\frac{t^2}{\xi^2- 2 \xi\cos t +1}}\, {\rm d}\xi .
\end{align*}
Notice that
$$\lim_{t \to 0}\frac{t^2}{\xi^2- 2 \xi\cos t +1}= \begin{cases} 0 & {\rm for} \, 0 \leq \xi < 1\\1& {\rm for} \,  \xi = 1;
\end{cases}
$$
thus 
$$\int_0^1\lim_{t \to 0}\frac{1}{\xi^{\beta_n} }\cdot \sqrt{\frac{t^2}{\xi^2- 2 \xi\cos t +1}}\, {\rm d}\xi=0.$$
It is easy to see that
$$\xi^2- 2 \xi\cos t +1\geq \sin^2 t,\qquad  \forall t, \xi ;$$
moreover, since  $\beta_n < c < 1$ by assumption, we get 
\begin{align*}
\frac{1}{\xi^{\beta_n} }\cdot \sqrt{\frac{t^2}{\xi^2- 2 \xi\cos t +1}}  \leq \frac{1}{\xi^{c} } \left|\frac{t}{\sin t}\right|\leq \frac{C}{\xi^{c}}
\end{align*}
in a neighborhood of 0. Since the function $ \xi \mapsto \frac{C}{\xi^{c}}$ is integrable in $(0,1)$, Lebesgue Theorem assures that, for every $n$, 
$$\lim_{t \to 0}\left|\frac{{\rm e}^{{\rm i}t}-1}{{\rm i}t}\right|\int_0^1\frac{1}{\xi^{\beta_n} }\cdot \sqrt{\frac{t^2}{\xi^2- 2 \xi\cos t +1}}\, {\rm d}\xi=\int_0^1\lim_{t \to 0}\frac{1}{\xi^{\beta_n} }\cdot \sqrt{\frac{t^2}{\xi^2- 2 \xi\cos t +1}}\, {\rm d}\xi=0.$$
This convergence is uniform in $n$ since
$$\sup_n \left|\frac{{\rm e}^{{\rm i}t}-1}{{\rm i}t}\right|\int_0^1\frac{1}{\xi^{\beta_n} }\cdot \sqrt{\frac{t^2}{\xi^2- 2 \xi\cos t +1}}\, {\rm d}\xi \leq \left|\frac{{\rm e}^{{\rm i}t}-1}{{\rm i}t}\right|\int_0^1\frac{1}{\xi^{c} }\cdot \sqrt{\frac{t^2}{\xi^2- 2 \xi\cos t +1}}\, {\rm d}\xi \to 0, \qquad t \to 0.$$
\end{proof}

\section{Discussion}
 In this section we discuss the constant $c_{F_n}$ that appears in Proposition \ref{integralconv}.  Recall that $c_{F_n}$  is the sum of $$m_{F_n}:=1-\alpha_n\gamma  \quad {\rm and} \quad  b_{F_n}:= \int_0^1 \frac{1}{u}\left(\frac{F_n(u)}{u}-  \alpha_n\right)\, {\rm d}u.$$
Observe that the function  $F_n$ at use, enters completely in the definition  of $b_{F_n}$, while it appears in the definition of $m_{F_n}$ only through $\alpha_n$.  Here we try to give an explanation of the presence of $\gamma,$ Euler's constant in the constant $m_{F_n}$. Although the $\gamma$ constant is  related to the standard uniform distribution (as we shall see this in the sequel), it appears that no matter the form of the distributions $F_n$ that are involved, the Euler's constant remains unchanged.

\medskip

	\noindent
	Let $Y_n= \frac{1}{U_n}$, where $U_n$ are i.i.d. with law $\mathcal{U}([0,1])$ (i.e. $F(x)\equiv x$) and  take $Z_n = \left\lceil\frac{1}{U_n} \right \rceil$. Then, for $k \geq 2$,
	$$P (Z_n=k ) = P\left(\left\lceil\frac{1}{U_n} \right \rceil=k\right)= P\left(k-1< \frac{1}{U_n}\leq k \right)= P\left(\frac{1}{k}\leq U_n < \frac{1}{k-1}\right)= \frac{1}{k-1}-\frac{1}{k}.$$
	Note that this construction of $Z_n$ is in agreement with the random variables that appear in  Corollary \ref{Cor2} with $\beta_n =0$  (i.e. the L\"uroth case). Denote
	$$\psi (t)= \psi_{Y_n}(t)= \int_0^1 {\rm e}^{{\rm i}\frac{t}{u}}\,{\rm d} u; \qquad  \phi (t) = \psi_{Z_n} (t) = \sum_{k \geq 2}{\rm e}^{{\rm i}t k}\left(\frac{1}{k-1}- \frac{1}{k}\right),$$
	and
	$$g (t)=  \frac{\psi (t) - {\rm e}^{{\rm i}t } }{{\rm e}^{{\rm i}t } -1};\qquad \tilde g (t)=  \frac{\phi (t) - {\rm e}^{{\rm i}t } }{{\rm e}^{{\rm i}t } -1}.$$
	We wish to calculate (with obvious notation)
	\begin{align*}
	c_2-\tilde c_2= \lim_{t \to 0} \left( g(t) + c_1 \log (1-{\rm e}^{{\rm i}t})\right)- \left( \tilde g(t) + \tilde c_1 \log (1-{\rm e}^{{\rm i}t})\right)= \lim_{t \to 0} \left( g (t)-\tilde g (t)\right),
	\end{align*}
since $c_1 =1$ (see the proof of Corollary \ref{Cor1}) and $\tilde c_1 =1$  (see the proof of Corollary \ref{Cor2}).
	Thus,
	$$c_2-\tilde c_2=\lim_{t \to 0} \frac{\psi (t) - {\rm e}^{{\rm i}t } }{{\rm e}^{{\rm i}t } -1}-\frac{\phi (t) - {\rm e}^{{\rm i}t } }{{\rm e}^{{\rm i}t } -1} = 
	\lim_{t \to 0} \frac{\psi (t) - \phi (t) }{ {\rm i}t }  .
	$$
	We can write
	$$\psi (t) = \sum_{k \geq 2}\int_{\frac{1}{k }}^\frac{1}{k-1} {\rm e}^{{\rm i}\frac{t}{u}}\,{\rm d} u = \lim_{n \to \infty } \sum_{k  = 2}^n\int_{\frac{1}{k}}^\frac{1}{k-1} {\rm e}^{{\rm i}\frac{t}{u}}\,{\rm d} u;$$
	$$\phi (t) = \sum_{k \geq 2} {\rm e}^{{\rm i}t k}\left(\frac{1}{k}- \frac{1}{k-1}\right)= \lim_{n \to \infty } \sum_{k  = 2}^n\int_{\frac{1}{k}}^\frac{1}{k-1} {\rm e}^{{\rm it}k}\,{\rm d} u. $$
	Hence,
	\begin{align*}&
	\lim_{t \to 0}\frac{\psi (t) - \phi (t) }{ {\rm i}t } = \lim_{t \to 0}\lim_{n \to \infty }\sum_{k  = 2}^n\int_{\frac{1}{k}}^\frac{1}{k-1} \frac{ {\rm e}^{{\rm i}\frac{t}{u}} -  {\rm e}^{{\rm it}k} }{{\rm i}t}\,{\rm d}u =\lim_{n \to \infty } \lim_{t \to 0}\sum_{k  = 2}^n\int_{\frac{1}{k}}^\frac{1}{k-1} \frac{ {\rm e}^{{\rm i}\frac{t}{u}} -  {\rm e}^{{\rm it}k} }{{\rm i}t}\,{\rm d}u \\&=\lim_{n \to \infty }\sum_{k  = 2}^n\int_{\frac{1}{k}}^\frac{1}{k-1}  \lim_{t \to 0}\frac{ {\rm e}^{{\rm i}\frac{t}{u}} -  {\rm e}^{{\rm it}k} }{{\rm i}t}\,{\rm d}u = \lim_{n \to \infty }\sum_{k  = 2}^n\int_{\frac{1}{k}}^\frac{1}{k-1}  \left(\frac{1}{u}-k\right)\,{\rm d}u   \\& =  \lim_{n \to \infty }\left( \sum_{k  = 2}^n\int_{\frac{1}{k}}^\frac{1}{k-1}  \frac{1}{u}\,{\rm d}u  - \sum_{k  = 2}^n k\left(\frac{1}{k-1}- \frac{1}{k}\right)\right)= \lim_{n \to \infty }\left( \sum_{k  = 2}^n\int_{k-1}^k  \frac{1}{u}\,{\rm d}u  - \sum_{k  = 2}^n k\left(\frac{1}{k-1}- \frac{1}{k}\right)\right )\\& =  \lim_{n \to \infty }\left(\int_1^{n} \frac{1}{u}\,{\rm d}u  -\sum_{k  = 2}^n  \frac{1}{k-1}\right)=     \lim_{n \to \infty }\left( \int_1^{n} \frac{1}{u}\,{\rm d}u  -\sum_{k  = 1}^{n-1}  \frac{1}{k}\right)
	\\& =     \lim_{n \to \infty }\left(\int_1^{n} \frac{1}{u}\,{\rm d}u  -\sum_{k  = 1}^{n}  \frac{1}{k}+ \frac{1}{n}\right) = -\gamma,
	\end{align*}
	by the definition of $\gamma$.  Note that the second equality can be verified by using Moore-Osgood Theorem (see \cite{T2012}, p. 140), while the third one by the bounded convergence theorem.
	Hence
	$$c_2-\tilde c_2  =  -\gamma=\lim_{n \to \infty }\left( \int_1^{n} \frac{1}{u}\,{\rm d}u -\sum_{k  = 1}^{n}  \frac{1}{k}\right) =\lim_{n \to \infty }\left(\int_1^{n} F\left(\frac{1}{u}\right) \,{\rm d}u -\sum_{k  = 1}^{n}  F\left(\frac{1}{k}\right)\right),$$
	where $F$ is the distribution function of the $\mathcal{U}([0,1])$.  
	Since $\tilde{c_2}=1$ (Corollary \ref{Cor2}) we obtain $c_2 = 1 -\gamma$, thus
	$$ \frac{m_{F_n}}{\alpha_n}  =\frac{1}{\alpha_n}  -\gamma =\frac{1}{\alpha_n}+\lim_{n \to \infty }\left( \int_1^{n} F\left(\frac{1}{u}\right) \,{\rm d}u -\sum_{k  = 1}^{n}  F\left(\frac{1}{k}\right)\right) $$
	This means that the constant $ \frac{m_{F_n}}{\alpha_n} $ is determined only by the distribution function $F(x) \equiv x$, i.e. no matter which is  the   $F_n$ at stake, this plays the same role as $F(x) \equiv x$.
	
\newpage

 \renewcommand{\theequation}{A-\arabic{equation}}
 \renewcommand{\thetheorem}{A\arabic{theorem}}
\setcounter{equation}{0}  
\setcounter{theorem}{0} 
\section*{APPENDIX}  
\begin{lemma}\label{lemma1}\sl
	Consider the two constants
	$$A=\int_0^1 \frac{\sin x -x}{x^2}  \, {\rm d}x; \qquad B= \int_1^{+ \infty } \frac{\sin x}{x^2} \, {\rm d}x.$$
	We have $$A+B = 1 -\gamma,$$
	where $\gamma $ is Euler's constant.
\end{lemma}
\begin{proof}
	We begin by transforming $A$. Denote
	$$A(x) = \int_0^x \frac{\sin t -t}{t^2}  \, {\rm d}t.$$
	For every $\epsilon >0$ and $x >0$
	\begin{align*}&
	\int_\epsilon^x \frac{\sin t -t}{t^2}  \, {\rm d}t= \int_\epsilon^x \frac{\sin t  }{t^2}  \, {\rm d}t -  \int_\epsilon^x \frac{1}{t }  \, {\rm d}t
	= -\frac{\sin t}{t}\Big|_\epsilon^x+ \int_\epsilon^x \frac{\cos t}{t}\, {\rm d}t  -  \int_\epsilon^x \frac{1}{t }  \, {\rm d}t \\&= 
	-\frac{\sin x}{x} +   \frac{\sin \epsilon } {\epsilon}+ \int_\epsilon^x \frac{\cos t-1}{t}\, {\rm d}t =  -\frac{\sin x}{x} +   \frac{\sin \epsilon } {\epsilon}- {\rm Cin} (x) +  {\rm Cin} (\epsilon),
	\end{align*}
	where 
	$$ {\rm Cin} (x) =\int _{ 0} ^x \frac{1-\cos t }{t}\, {\rm d}t.$$
	Therefore,
	$$A(x) = \lim_{\epsilon \to 0}\left(  -\frac{\sin x}{x} +   \frac{\sin \epsilon } {\epsilon}- {\rm Cin} (x) +  {\rm Cin} (\epsilon)\right) = 1 - \frac{\sin x}{x} - {\rm Cin} (x). $$
	It is well known (see \cite{AS}, formula 5.2.2) that
	$$ {\rm Cin} (x) = \gamma + \log x -  {\rm Ci} (x),\quad |\arg x| < \pi,$$
	where $\gamma$ is Euler's constant and $${\rm Ci} (x) = - \int_x ^\infty \frac{\cos t}{t} \, {\rm d}t ,\quad |\arg x| < \pi $$
	is the {\it cosine integral function.}
	We obtain
	$$A(x) =  1 - \frac{\sin x}{x} - \gamma - \log x +  {\rm Ci} (x),\quad |\arg x| < \pi ,$$
	and we conclude that
	$$A=A(1) = 1-\sin 1 - \gamma  +  {\rm Ci} (1).$$
	
	\noindent For the calculation of the constant $B$, we employ  formula 3, \$ 3.761, page 436 in  \cite{GR},  namely
	$$\int_1^\infty \frac{\sin (ax)}{x^{2n}}\, {\rm d}x= \frac{a^{2n-1}}{(2n-1)!}\left(\sum_{k=1}^{2n-1}\frac{(2n-k-1)!}{a^{2n-k}}\sin \left(a+ (k-1)\frac{\pi}{2}\right)+ (-1)^n {\rm Ci} (a)\right).$$
	Substituting $a=1$, $n=1$ we find
	$$B= \sin 1-  {\rm Ci} (1).$$
	Thus,
	$$A+B =  1-\sin 1 - \gamma  +  {\rm Ci} (1) + \sin 1-  {\rm Ci} (1) =1- \gamma \approx 0, 423.$$
\end{proof}

\begin{lemma}\label{approx1} \sl Under the assumptions of Theorem \ref{generalres} we have that
	\begin{equation*}
	\sum_{k=1}^n \left(\log (\phi_{ k} (ta_{k,n})) -(\phi_{ k}(ta_{k,n})-1) \right) \to 0, \qquad n \to \infty.
	\end{equation*}
\end{lemma}
\begin{proof} Recall that
	$$\phi _{ k}(ta_{k,n})-1 = \left({\rm e}^{{\rm i}ta_{k,n}}-1\right)\left(1+ g_{ k}(ta_{k,n})\right).$$
	By assumption \eqref{as1}, there exists $\delta>0$ such that, for any $k$, the function $t \mapsto  \frac{|g_{ k}(t)|}{|\log|t||^\eta}$ is bounded by a constant $M$ (not depending on $k$) for $|t|< \delta.$ By \eqref{equazione2} there exists $n_0$ such that $|t|a_{k,n}\leq |t| m_n < \delta$ for $n>n_0$ and $k=1, \dots, n$. Thus
	\begin{align*}&
	\left|\phi_{ k}(ta_{k,n})-1\right|^2= \left|{\rm e}^{{\rm i}ta_{k,n}}-1\right|^2\left|1+ g_{ k}(ta_{k,n})\right|^2\leq 2\left|{\rm e}^{{\rm i}ta_{k,n}}-1\right|^2\left(1+ \left|g_{ k}(ta_{k,n})\right|^2\right)\\ 
	&\leq 2\left|{\rm e}^{{\rm i}ta_{k,n}}-1\right|^2 \left(1+ M\left|\log(|t| a_{k,n})\right|^{2 \eta} \right) \leq C t^2 a^2_{k,n}\left(1+  M\left|\log(|t| a_{k,n})\right|^{2 \eta}\right)\\ 
&\leq C t^2 m^2_{n}\left(1+  M\left|\log (|t| m_{n})\right|^{2 \eta}\right),
	\end{align*}
	for sufficiently large $n$ and for $k =1, \dots, n$; the last inequality holds true since  the function $ x \mapsto x^2 \left|\log |x|\right|^{2 \eta}$ is increasing in a neighborhood of 0 and the relation
	$$\lim_{z \to 0} \frac{{\rm e}^z-1}{z}=1$$ 
	implies that, in a neighborhood of $0$,  $$\left|{\rm e}^z-1\right| = \left|\frac{{ \rm e}^z-1 }{z}\right|\cdot |z|\leq C |z|. $$
	Hence, by \eqref{equazione2}, we deduce
	$$\lim_{n \to \infty}\sup_k\left|\phi_{ k}(ta_{k,n})-1\right|=0.$$
	Using Lemma \ref{linearizzazione}, fix $c$ and take $n$ large enough ($n> n_0$) in order that $ \sup_k\left|\phi_{ k}(ta_{k,n})-1\right| < c$. Then Lemma \ref{linearizzazione} applies with $z = (\phi_{ k}(ta_{k,n})-1) $, ($k =1, 2, \dots, n$ ) and we find
\begin{align*}&
	\left|\sum_{k=1}^n \left\{\log (\phi_{ k} (ta_{k,n})) -(\phi_{ k}(ta_{k,n})-1) \right\}\right| \leq\sum_{k=1}^n  \left|\log (\phi_{ k}(ta_{k,n})) -(\phi_{ k}(ta_{k,n})-1)  \right| \\
	&\leq  m(c) \sum_{k=1}^n  \left|\phi_ {k}(ta_{k,n})-1\right|^2
	\leq 
	C\sum_{k=1}^n   a^2_{k,n}\left(1+  M\left|\log(|t| a_{k,n})\right|^{2  \eta} \right)\\
	& = C\left(\sum_{k=1}^n   a^2_{k,n} +  \sum_{k=1}^n   a^2_{k,n}\left|\log (|t| a_{k,n})\right|^{2  \eta}\right)\leq C \left(m_n \sum_{k=1}^n   a_{k,n}+ m_n \left|\log (|t|m_n) \right|^{2  \eta}\sum_{k=1}^n   a_{k,n} \right)\\ 
& \leq C \left(m_n + m_n  \left|\log(|t|m_n)  \right|^{2  \eta}  \right)\to 0, \qquad n \to \infty. 
	\end{align*}	
	the last inequality holds true by \eqref{equazione1} and because the function $ x \mapsto x  \left|\log |x|\right|^{2 \eta}$ is increasing in a  neighborhood of 0.	
\end{proof}

\begin{lemma}
	\label{approx2} \sl Under the assumptions of Theorem \ref{generalres} we have 
	$$\left|\sum_{k=1}^n\left({\rm e}^{{\rm i} ta_{k,n}}    -1 -{\rm i}ta_{k,n}- \frac{1}{2}({\rm i} ta_{k,n} )^2	\right)\right| \to 0, \qquad n \to \infty.$$
\end{lemma}
\begin{proof}
	Recall that
	\begin{equation} \label{equazione5}
	\lim_{z \to 0}\frac{{\rm e}^{z }-1 - z- \frac{1}{2}z^2}{z^3}= \frac{1}{6} .
	\end{equation}
	This implies  that, in a neighborhood of 0, say $|z|<\delta$,
	$$\left|{\rm e}^{z }-1 - z- \frac{1}{2}z^2\right|= \left|\frac{{\rm e}^{z }-1 - z- \frac{1}{2}z^2}{z^3}\right|\cdot |z|^3\leq C|z|^3. $$
	Let $n_0$ be such that $|t|m_n < \delta$  for every $n> n_0$ ($n_0$ exists by   assumption \eqref{equazione2}). Then $|t|a_{k,n} \leq |t| m_n < \delta$  for $k=1, \dots, n$, and, by the above,
	\begin{align*}
	\left |{\rm e}^{{\rm i} ta_{k,n}}    -1 -{\rm i}ta_{k,n}- \frac{1}{2}({\rm i} ta_{k,n} )^2	 \right|\leq C  |t|^3a_{k,n}  ^3.
	\end{align*}
	Therefore,
	\begin{align*}
	\left|\sum_{k=1}^n\left({\rm e}^{{\rm i} ta_{k,n}}    -1 -{\rm i}ta_{k,n}- \frac{1}{2}({\rm i} ta_{k,n} )^2	\right)\right| \leq    |t|^3  C \sum_{k=1}^n a^3_{k,n} 
	\leq  |t|^3  C  m^2_n  \sum_{k=1}^n a _{k,n} \to 0, \qquad n \to \infty,
	\end{align*}
	by assumption \eqref{equazione1}, which concludes the proof.
\end{proof}

 \begin{lemma}\label{approx3} \sl Under the assumptions of Theorem \ref{generalres}, 
	$$\sum_{k=1}^n a_{k,n}c_{1,k} \left(\log (1-{\rm e}^{{\rm i}t a_{k,n}})-\log (-{\rm i}t a_{k,n})\right)\ \to 0, \qquad n \to \infty.$$
\end{lemma}
\begin{proof}
Let $z = {{\rm i}\tau}$. For $\tau$ in a neighborhood of $0$ we can write  
	\begin{align*}&
	\log (1-{\rm e}^z)-\log (-z) = \log \left( \frac{1-{\rm e}^z}{-z}\right) = \log \left( \frac{ {\rm e}^z-1}{ z}\right)= \log \left( 1+ \frac{ {\rm e}^z-1 -z}{ z}\right).
	\end{align*}
	Notice that
	$$\lim_{z \to 0} \frac{ {\rm e}^z-1 -z}{ z}=0,$$
	so that for fixed $c$  there exists $\delta_1 >0$ such that, for $|z|< \delta_1$, we have
	$$\left| \frac{ {\rm e}^z-1 -z}{ z}\right|  <c.$$
	It follows from   Lemma \ref{linearizzazione} that,  for $|z|< \delta _1,$
	$$\left|\log (1-{\rm e}^z)-\log (-z) -\frac{ {\rm e}^z-1 -z}{ z}\right|\leq m(c) \left| \frac{ {\rm e}^z-1 -z}{ z}\right| ^2. $$
	Now, since
	$$\lim_{z \to 0} \frac{ {\rm e}^z-1 -z}{ z^2}= \frac{1}{2}, $$
	there exists $\delta_2$ such that, for $|z|< \delta_2$ , the function $$z \mapsto \frac{ {\rm e}^z-1 -z}{ z^2}$$ is bounded. Thus,
		$$ \left| \frac{ {\rm e}^z-1 -z}{ z}\right| = \left| \frac{ {\rm e}^z-1 -z}{ z ^2}\right| \cdot |z|\leq C |z| $$ 
		which leads to
		$$\left|\log (1-{\rm e}^z)-\log (-z) -\frac{ {\rm e}^z-1 -z}{ z}\right|\leq C |z|^2$$
		if $|z|< \delta:= \delta_1 \wedge \delta_2.$ Therefore, if $|z|< \delta$, we get 
		\begin{align}\label{eq6}
		\left|\log (1-{\rm e}^z)-\log (-z) \right|\leq \left|\log (1-{\rm e}^z)-\log (-z) -\frac{ {\rm e}^z-1 -z}{ z}\right|+ \left| \frac{ {\rm e}^z-1 -z}{ z}\right| \leq C(|z|+ |z|^2).
		\end{align}
		Let $n_0$ be such that,  for $n > n_0$, $ |{\rm i} t m_n| =|t|m_n< \delta$  ($n_0$ exists by  assumption  \eqref{equazione2}). Then $|{\rm i} |t| a_{k,n}|= |t|a_{k,n} \leq |t|m_n< \delta$, so that, for every  $n> n_0$ and for $k=1, \dots, n$ we have

		\begin{align*}
		&
		\left|\sum_{k=1}^n a_{k,n} c_{1,k} \left(\log (1-{\rm e}^{{\rm i}t a_{k,n}})-\log (-{\rm i}t a_{k,n})\right) \right| \\
		&\leq C \sum_{k=1}^n  a_{k,n}c_{1,k} ( |t| a_{k,n}+ |t|^2 a^2_{k,n})
		\leq  C( |t| m_n + |t|^2 m_n^2)\sum_{k=1}^n a_{k,n} c_{1,k}  \to 0, \qquad n \to \infty
		\end{align*}
		due to \eqref{equazione2} and \eqref{equazione4}.
	\end{proof}

\begin{lemma} \sl \label{approx4}Under the assumptions of Theorem \ref{generalres} we have,
	\[
	\sum_{k=1}^{n}a_{k,n}\left(g_k(ta_{k,n}) + c_{1,k} \log(1-e^{ita_{k,n}}) -  c_{2,k} \right)\to 0\mbox{ as }n\to \infty.
	\]
\end{lemma}
\begin{proof}
	Note that
	\begin{align*}
	&\left| \sum_{k=1}^{n}a_{k,n}\left(g_k(ta_{k,n}) + c_{1,k} \log(1-e^{ita_{k,n}})- c_{2,k} \right)\right|
	\leq  \sum_{k=1}^{n}a_{k,n}\left|g_k(ta_{k,n}) + c_{1,k} \log(1-e^{{\rm i}ta_{k,n}})- c_{2,k} \right|.
	\end{align*}
	We prove that the last quantity goes to 0, as $n \to \infty.$    Assumption \eqref{as2} means that, for every  $\epsilon >0$, there exists $\delta >0 $ such that
	$$\big|g_k(t) +  c_{1,k} \log \big(1-e^{{\rm i}t}\big) - c_{2,k} \big|< \epsilon$$
	for every $t$ with $|t|< \delta $ and for every $k$.
	By assumption \eqref{equazione2},  there exists $n_0$ such that, for every $n>n_0$, $$\sqrt{2(1- \cos (tm_n) )}= \big|{\rm e}^{{\rm i}t m_n} -1\big|< \delta\qquad {\rm and } \qquad | t m_n|=|{\rm i}t m_n|< \rho,$$ 
where $\rho>0$  is such that the function $x \mapsto 1 - \cos x$ is increasing  for $|x|< \rho$; hence, for $n>n_ {0}$ and  $k=1, \dots, n$
	$$\big|{\rm e}^{{\rm i}t a_{k,n}} -1\big| =\sqrt{2(1-\cos (t a_{k,n}))}  \leq\sqrt{2(1-\cos (tm_n ))}= \big|{\rm e}^{{\rm i}t m_n} -1\big|<   \delta ,$$
	and thereby
	$$ \left| g_{ k}(ta_{k,n}) + c_{1,k} \log(1-e^{{\rm i}ta_{k,n}})- c_{2,k}\right|< \epsilon.$$
	As a consequence, for $n > n_{0} $,
	\begin{align*}
	\sum_{k=1}^{n}a_{k,n}\left| g_{ k}(ta_{k,n}) + c_{1,k}\log(1-e^{{\rm i}ta_{k,n}})-c_{2,k} \right|\leq \epsilon\sum_{k=1}^{n}a_{k,n},
	\end{align*}
	and convergence to 0 is ensured by the arbitrariness of $\epsilon$ and assumption \eqref{equazione1}.	
\end{proof}

\begin{lemma} \sl
	\label{approx5}Under the assumptions of Theorem \ref{generalres} we have,
	$$\sum_{k=1}^n \left( e^{{\rm i}ta_{k,n}} - 1 - {\rm i}ta_{k,n} - \frac{1}{2} ({\rm i}ta_{k,n})^2 \right) g_k (ta_{k,n} ) \to 0, \qquad n \to \infty.$$
\end{lemma}
\begin{proof}
	By assumption \eqref{as1} and relation \eqref{equazione5} we have
	\begin{align*}&
	\left|\sum_{k=1}^n \left( e^{{\rm i}ta_{k,n}} - 1 - {\rm i}ta_{k,n} - \frac{1}{2} ({\rm i}ta_{k,n})^2 \right) g_{ k}  (ta_{k,n} ) \right|\\ 
	&\leq 
	\sum_{k=1}^n \left| e^{{\rm i}ta_{k,n}} - 1 - {\rm i}ta_{k,n} - \frac{1}{2} ({\rm i}ta_{k,n})^2 \right|\cdot \left| g_{ k}  (ta_{k,n} ) \right|\\ 
	&=
	\sum_{k=1}^n  \left|\frac{  e^{{\rm i}ta_{k,n}} - 1 - {\rm i}ta_{k,n} - \frac{1}{2} ({\rm i}ta_{k,n})^2  }{ ({\rm i}ta_{k,n})^3}\right| \cdot \left| g_{ k}  (ta_{k,n} ) \right| \cdot |t|^3 a^3_{k,n}\\
	& \leq C{ |t|^3} \sum_{k=1}^n a^3_{k,n}|\log(|t|  a_{k,n})|^{\eta}\leq C{ |t|^3}\left(\max_{1 \leq k \leq n}a^2_{k,n} |\log  (|t| a_{k,n} )|^{\eta}\right){ \sum_{k=1}^n a_{k,n}}\\&  
	\leq C {  |t|}\left(m^2_{n} |\log^{\eta} (m_{n} )|\right)\sum_{k=1}^n a_{k,n}\leq C  { |t|}\left(m^2_{n} |\log (m_{n} )|^{\eta}\right) \to 0, \qquad n \to \infty,
	\end{align*} 
	(notice that the function $x \mapsto x  ^2 \left|\log |x| \right|^{\eta}$ is increasing in the neighborhood of 0).
\end{proof}

\end{document}